\documentclass[twoside, 10pt]{article}
\usepackage{xcolor}
\usepackage{cite}
\usepackage{amsmath}
\usepackage{paralist}
\usepackage{amssymb,mathrsfs,psfrag,graphicx}
\usepackage{hyperref}
\usepackage{amsfonts}
\usepackage{indentfirst}
\usepackage{latexsym}
\newtheorem{theorem}{Theorem}[section]
\newtheorem{remark}{Remark}[section]
\newtheorem{definition}{Definition}[section]
\newtheorem{lemma}[theorem]{Lemma}

\newcommand{\mr}{\mathbb{R}}

\renewcommand{\div}{ {\rm div }  }

\newcommand{\na}{\nabla }

\newcommand{\bl}{\begin{lemma}}
	\newcommand{\el}{\end{lemma}}
\newcommand{\et}{\end{theorem}}

\newcommand{\la}{\label}

\newcommand{\Om}{\Omega}

\newcommand{\bn}{\begin{eqnarray}}
\newcommand{\en}{\end{eqnarray}}
\newcommand{\bnn}{\begin{eqnarray*}}
\newcommand{\enn}{\end{eqnarray*}}

\newcommand{\bnnn}{\begin{eqnarray*}}
\newcommand{\ennn}{\end{eqnarray*}}

\newcommand{\ba}{\begin{aligned}}
\newcommand{\ea}{\end{aligned}}
\newcommand{\be}{\begin{equation}}
\newcommand{\ee}{\end{equation}}

\def\norm[#1]#2{\|#2\|_{#1}}

\newcommand{\diver}{\operatorname{div}}

\newcommand{\average}[1]{\frac{1}{|#1|}\int_{#1}}

\def\la{\label}

\def\na{\nabla}

\makeatletter
\@addtoreset{equation}{section}
\makeatother
\title{ Global existence and uniqueness of Lions’ weak solutions  for the 2D inhomogeneous Navier-Stokes equations  in  bounded domains }
\author{Hongyu Wang$^a$\thanks{
	Email addresses: {hyuwang\_a@163.com} (H. Y. Wang).} \\[3mm]  a. School of Mathematics and computer Sciences,\\
Nanchang University, Nanchang,  330031, P. R. China.} 

\date{ }

\begin{document}
\maketitle
\begin{abstract}
		The global well-posedness   for Lions' weak solutions  to  the  inhomogeneous incompressible  Navier-Stokes  equations  in two-dimensional bounded domains  is concerned.  When the initial   
density $0\leq \rho_0\in L^\infty(\Om)$ (allowing for vacuum) and initial velocity   $u_0\in L^2_\sigma(\Omega)$,  the global existence, regularity, and exponential decay-in-time rates  of weak solutions are established via  a priori estimates, which rely crucially  on the duality of BMO and  Hardy space for the pressure term, and the classical Sobolev extension theorem. In particular, for the non-vacuous case  where $0< \rho_0\in L^\infty(\Om)$ is bounded away from zero,  we further prove the uniqueness of weak solutions in the sense of approximation by strong solutions,  based on the key estimate \(t^{1/2}\nabla u\in L^2(0,T;L^\infty(\Omega))\), which essentially follows from  a delicate localized point-wise bound for \(\nabla u\) near the boundary.  The latter bound is derived by the  Green functions estimate  for   the nonhomogeneous Dirichlet Stokes problem,  and   the uniform \(C^1\) boundary  estimate on the remaining  homogeneous Stokes problem,  through  a   Stokes decomposition on suitably   constructed  local \(C^2\) domains.
\\

\end{abstract}

\section{Introduction}
The two-dimensional   nonhomogeneous incompressible Navier-Stokes equations  read as follows:
\begin{equation}
	\left\{\begin{array}{l}
		\rho_{t}+\operatorname{div}(\rho u)=0, \\
		(\rho u)_{t}+\operatorname{div}(\rho u \otimes u)- \mu \Delta u+\nabla P=0,\\
		\operatorname{div} u=0.
		\label{11}
	\end{array}\right.
\end{equation}
where $t\geq0, x\in \Omega\subset\mathbb{R}^2,$ $\Omega$ is a bounded connected  smooth domain. Moreover, $\rho, u$, and $P$ represent the density, velocity, and  pressure of the fluid, respectively.  The constant viscosity $\mu=1$  for simplicity.  In addition, system \eqref{11} is imposed with the following initial conditions
\begin{equation}
	\rho(x, 0)=\rho_{0}(x),\quad \rho u(x, 0)=\rho_{0} u_{0}(x), \quad x\in \Omega
	\label{13}
\end{equation}
and no-slip boundary condition
\begin{equation}
	u = 0, \quad x\in \partial \Omega.
	\label{14}
\end{equation}

The well-posedness of incompressible viscous fluid flows \eqref{11} has been extensively investigated in the literature. When the density is  constant (e.g., $\rho \equiv 1$), system \eqref{11} reduces to the classical homogeneous  Navier--Stokes  equations. In 1934, Leray \cite{Leray1934} established  the existence of global weak solutions (termed as turbulent solutions) for arbitrary initial velocity  $u_0 \in L^2(\mathbb{R}^3)$ satisfying $\operatorname{div} u_0 = 0$.   Classical  two-dimensional existence  and uniqueness   results   were    obtained by Leray  \cite{Leray1933, Leray1934-1}, Ladyzhenskaya \cite{Ladyzhenskaya1958}, and Lions and Prodi (see \cite{Prodi1960}). Later, for the three-dimensional Dirichlet problem, the critical regularity theory for strong solutions was initiated by Fujita and Kato \cite{KatoFujita1962}, who demonstrated local well-posedness for large initial data and global well-posedness for sufficiently small initial data in  the critical   fractional domain of the Dirichlet Stokes operator.  Moreover, for small data in $\dot{H}^{1/2}(\mathbb{R}^3)$, Chemin \cite{Chemin2005} gave an energy estimate showing that the critical norm is nonincreasing; while, by virtue of   semigroup   theory, 
Bedrossian and Vicol \cite{Bedrossian}   stated a modern proof of
     local well-posedness
 with large data and   global well-posedness
with   small data  in  $\dot{H}^{1/2}(\mathbb{R}^3)$.

For nonhomogeneous flows (i.e., the density is not constant), the  mathematical  analysis becomes  more complicated.    For bounded initial density bounded away from zero and {  square-integrable solenoidal initial velocity,} Kazhikhov \cite{Kazhikov1974} proved global existence of weak solutions in bounded { three-dimensional} domains. Later,  many researchers have focused on  the uniqueness of solutions which is fundamentally characterized by a balance between the regularity of  $u_0$ and the assumptions on $\rho_0$. For {  strong} {  solutions,} Ladyzhenskaya and Solonnikov \cite{Ladyzhenskaya1975} proved global unique solvability in 2D (and local in 3D) under Dirichlet boundary conditions {  for} {  divergence-free} {  initial velocity $u_0\in W^{2-2/p,p}(\Omega)$ with zero boundary trace, $p>n$,} and strictly { positive density $\rho_0\in C^1(\overline\Omega)$.} Then, to relax the velocity regularity to scaling-invariant levels, substantial progress was made within the critical Besov framework (extending the Fujita--Kato theory). In spaces such as $\dot{B}_{p,1}^{-1+n/p}(\mathbb{R}^n)$, the summability index $1$ naturally ensures $\nabla u \in L^1(0,T; L^\infty(\mr^n))$, which is crucial for  uniqueness. Building on this approach, global well-posedness {  under suitable smallness conditions} was established in \cite{Danchin2003}, \cite{A_2007} and \cite{DanchinMucha2012}, {  for} positive initial density (no vacuum) {  with} {  suitable} {  critical} {  Besov} {  regularity} or small {  multiplier} {  perturbations of} a positive {  constant.} {  Earlier} {  results} {  for} {  general} {  bounded} {  density,} { including} discontinuous or piecewise constant {  density,} {  often} required {  higher} {  velocity} {  regularity.} Danchin and Mucha \cite{Danchin_2013} proved local-in-time  uniqueness  in bounded domain with $u_0 \in H^2(\Omega)$ in dimension $n=2,3$. For the whole space, the  global well-posedness for bounded density (strictly away from zero) was established in \cite{pzz_2013}  with $u_0 \in H^s(\mathbb{R}^2)$ ($s>0$), or  $u_0 \in H^1(\mathbb{R}^3)$  with $\|u_0\|_{L^2(\mr^3)}\|\na u_0\|_{L^2(\mr^3)}$ sufficiently small. The same global existence and uniqueness hold for $u_0 \in H^s(\mathbb{R}^3)$ ($s>1/2$) with $\|u_0\|_{\dot H^{1/2}(\mathbb{R}^3)}$ sufficiently small (see \cite{wei_2025}). Moreover, Danchin and Wang \cite{Dan_2023} established the uniqueness when $\|\rho_0-1\|_{L^\infty(\mathbb{R}^n)}$ is sufficiently small with $u_0 \in \dot{B}_{p,1}^{-1+n/p}(\mathbb{R}^n)$ ($p \in (1,n), n=2,3$), { with a sufficiently small critical velocity norm in dimension three} and { the} additional condition that $u_0 \in L^2(\mathbb{R}^3)$ for $p \in (2,3)$.

When the initial density is allowed to vanish, Simon \cite{Simon_1990}  established the global existence of weak solutions for  $0 \leq \rho_0 \in L^\infty(\Om)$ and {$u_0 \in L^2_\sigma(\Om)$} in bounded domains. In  \cite{Lions_1996},  Lions further extended the existence  to more general settings, including density-dependent viscosity, where  the density equation holds true in a renormalized sense.
 However, {for general initial densities allowing vacuum,} uniqueness of { Lions'} weak solutions remains a longstanding open problem, since  $u_0\in L^2(\Om)$ fails to ensure the  critical regularity  $\nabla u \in L^1(0,T; L^\infty(\Om))$, the key obstruction to  standard Eulerian uniqueness proofs. Hence, uniqueness (or weak-strong uniqueness) has traditionally relied on   strong solution theory, requiring  additional compatibility conditions and higher regularity on the initial data, see \cite{Huang_2013, Huang_2014}. Even when compatibility conditions were  removed or relaxed,  uniqueness still needs to elevate $u_0$ beyond $L^2(\Om)$ (typically requiring $\nabla u_0 \in L^2(\Om)$ or higher),  with extra technical constraints on density.  L\"{u}, Shi, and Zhong \cite{Lu_2018} obtained the global existence and uniqueness of strong solutions in $\mathbb{R}^2$  under the condition
 $\sqrt{\rho_0}u_0\in L^2(\mr^2)$ and   $\nabla u_0 \in L^2(\mr^2)$ coupled with specific density constraints.
  For 3D, it should be noted  that
    He, Li, and L\"{u} \cite{Li_2021} established exponential stability of strong solutions, with higher  velocity
  regularity ($\nabla u_0 \in L^2(\mathbb{R}^3)$ and smallness in fractional Sobolev spaces) and additional density restrictions. {  Remark 1.4 in \cite{Li_2021} discusses strong-weak uniqueness.}

In recent years, the density patch problem ($\rho_0=a\mathbf{1}_D+b$ for some $a>b\geq 0$ and some bounded open set  $D\subset\Om$)   raised by 
Lions in \cite{Lions_1996},  has attracted much attention.  The  interface regularity for initial velocity in non-critical space was studied in  {  \cite{GancedoEtAl2025, gancedo2018, zp_2016, zp_2019}}, and { related strong-solution results also allow general bounded nonnegative} initial { density.} For {  simply connected} bounded {  domains} and the torus, Danchin and Mucha \cite{DanchinMucha2019} proved global existence and uniqueness of strong solutions with $0\leq\rho_0\in L^\infty(\Om)$  and {  $u_0\in H^1_{0,\sigma}(\Omega)$ in the Dirichlet case, with the corresponding periodic space on the torus,} using Lagrangian {  coordinates.}  Later, these results were  extended to unbounded domains by Prange and Tan \cite{jin_2024} if $\sqrt{\rho_0}u_0 \in L^2(\mr^2)$ and $\nabla u_0 \in L^2(\mr^2)$ with additional far-field and compatibility assumptions. Recently, Hao, Shao, Wei, and Zhang \cite{wei_patch_2024} constructed global strong solutions by removing these extra assumptions,  and  established the   weak-strong uniqueness for compactly supported density or density with uniform local mass.  Moreover,  Danchin \cite{Danchin2025} demonstrated the  global existence of weak solutions for initial data  $0<\rho_0\in L^{\infty}(\mr^2)$ bounded away from zero and {  $u_0 \in L^2_\sigma(\mr^2)$}, however, uniqueness still required $u_0$ to belong to a suitable {  density-dependent} subspace of {  $L^2_\sigma(\mr^2)$} resembling $\dot{B}^0_{2,1}(\mr^2)$. More recently, a major breakthrough comes from  Hao, Shao, Wei, and Zhang  \cite{wei_2025},  establishing the uniqueness of Lions' weak solutions,  when $\rho_0\in L^\infty(\mathbb{R}^2)$ is bounded away from zero and $u_0 \in L^2(\mathbb{R}^2)$.
 
{  In} the present paper, {  for} $0\leq \rho_0\in L^\infty(\Om)$ (allowing vacuum) and $u_0 \in L^2_\sigma(\Omega)$, we  prove the global existence, regularity, and exponential decay-in-time rates of weak solutions. Moreover,  when  $0<\rho_0\in L^\infty(\Om)$ is bounded   away from zero, the  uniqueness  of weak solutions  is established by combining a stability estimate and an approximation via strong solutions.

Before  stating our main results,  we explain some notation and conventions  used throughout this paper. 

For $k\in\mathbb{Z}$, $1 \leq p\leq\infty$, the  Lebesgue  and Sobolev spaces are defined as follows:
\begin{equation*}
	\left\{
	\begin{array}{l}
		\|f\|_{L^p}=\|f\|_{L^p(\Omega)},\quad \|f\|_{W^{k,p}}=\|f\|_{W^{k,p}(\Omega)}, \\[2mm] \|f\|_{H^k}=\|f\|_{W^{k,2}(\Omega)},\quad	C_{0, \sigma}^{\infty}(\Omega)=\left\{f \in [C_{0}^{\infty}(\Om)]^2:  \operatorname{div} f=0 \right\},\\[2mm] L^2_\sigma=\overline{C_{0, \sigma}^{\infty}(\Om)}\,\,\text{closure in the norm of } L^{2},\\[2mm] H_{0, \sigma}^{1}=\overline{C_{0, \sigma}^{\infty}(\Om)} \,\,\text{closure in the norm of } H^{1}.
	\end{array}
	\right.
\end{equation*}
And for $1<q<\infty$, $1/q+1/{q'}=1$, $W^{-1,q}(\Om)$ denotes the dual space to $W^{1,q'}_0(\Om)$, and the duality pairing between $W^{-1,q}(\Omega)$ and $W^{1,q'}_0(\Omega)$ is denoted by $\langle\cdot,\cdot\rangle$.

Throughout the paper,  we use $C$ to denote a positive generic constant which may vary in different cases, and  $C(\alpha)$  to emphasize that $C$ depends on $\alpha$.   Moreover,   we set $B(x, r)$ as  the {  open ball} with center $x$ and radius $r$,
and also for an integrable function $f$,  
\be\nonumber
f_\Om \triangleq \frac{1}{|\Om|}\int_{\Om}f dx,
\ee
as the mean value of $f$ over $\Omega$.

Now, we introduce the definition of weak solutions.
\begin{definition}\la{d1}{\rm We say that $(\rho, u)$ is a global weak solution to problem \eqref{11}-\eqref{14}, if for any $T>0$, the following hold.
		\begin{description}
		\item(i) the density and the velocity satisfy  \[
		\rho\ge0,\quad
		\rho\in L^\infty(\Omega\times(0,T))\cap C([0,T);L^p(\Omega))
		\quad\text{for any }1\le p<\infty,
		\]
		\[
		u\in L^2(0,T;H^1_{0,\sigma}(\Omega)),\quad
		\rho|u|^2\in L^\infty(0,T;L^1(\Omega)),
		\]
		\[
		\mathbb P(\rho u)\in C([0,T];L^2_w(\Omega)),\quad
		\mathbb P (\rho u)(0)=\mathbb P (\rho_0u_0),
		\]where  	 $\mathbb P$   denotes the orthogonal  projection in $[L^2(\Om)]^2 $ onto {  $L^2_\sigma(\Omega)$,}  and $L^2_w$ means  the space \(L^2(\Om)\) endowed with its weak topology.
		\end{description}
		\begin{description}
			\item(ii) $(\rho, u)$ solves  $\eqref{11}$  in the sense of distribution, that is, for any  test functions $\varphi\in  C_0^{\infty}(\Omega \times[0, T)),$
			\begin{equation}\la{0}
				\int_{\Om}\rho_0\varphi(\cdot,0)dx+\int_0^T\int_{\Om}(\rho \varphi_t+\rho u\cdot\na \varphi) dxdt=0,
			\end{equation}
			and
			$$\int_0^T\int_{\Om}u\cdot\na \varphi dxdt=0,$$
			Moreover, 	for any  test functions  $g \in [C_0^{\infty}(\Omega \times[0, T))]^2 $ with $\operatorname{div} g=0$,
			\begin{equation}\la{00}
			\int_{\Om} \rho_0u_0 \cdot g(\cdot, 0) d x+\int_0^T\int_{\Om}\left(\rho u\cdot g_t+\rho u\otimes u:\na g\right) dxdt= \int_0^T\int_{\Om}\na u:\na g dxdt.
			\end{equation}
		\end{description}	
		
		\begin{description}
			\item(iii) $\left(\rho, u\right)$  satisfies the energy inequality
			\begin{equation}\la{y}
				\left(\frac{1}{2}\int_{\Omega} \rho\left|u\right|^2 d x\right)(t)+\int_0^t \int_{\Omega} |\na u|^2 d x d s \leq \frac{1}{2}\int_{\Omega}{\rho_0}|u_0|^2 dx,
			\end{equation}
			for a.e. $ t \in(0, T).$
	\end{description}	}
\end{definition}

Next, we are in a position to state our main results as follows.
\begin{theorem}\label{tm1} For given positive constant $\bar{\rho},$
	assume that the initial data $\left(\rho_{0}, u_{0}\right)$ satisfy 
	\begin{equation}
		0 \leq \rho_{0} \leq \bar{\rho}, \quad u_{0} \in L^2_{\sigma}.
		\label{15}
	\end{equation}	
	Then problem \eqref{11}-\eqref{14} admits a   global weak solution $(\rho, u, P)$ (with $\int_\Omega Pdx=0$)  satisfying for any $0<\tau<T<\infty$,
	\begin{equation}\la{150}
		\left\{\begin{array}{l}
		0 \leq \rho\leq\bar\rho,\,\,	0 \leq \rho \in C\left([0, T]; L^p\right), \,\forall  p\in[1,\infty),    \\[2mm]
			 u \in L^{2}(0, T;H_0^1), \,\, \na u\in L^\infty(\tau,T;L^2),\\[2mm]
			\sqrt{\rho} \dot{u} \in L^2\left(\tau, T; L^2\right)\cap  L^{\infty}\left(\tau, T; L^2\right),\\[2mm]
			\nabla \dot{u} \in  L^2\left(\tau, T; L^2\right), \,\, t^{\frac12}\na u \in L^2\left(0, T; L^\infty\right),\\[2mm]
			(\nabla u, P )\in  L^{\infty}\left(\tau, T; L^p\right)\cap L^2\left(\tau, T; W^{1, p}\right), \,\forall p\in(2,\infty),
		\end{array}\right.
	\end{equation}
where $\dot{u}\triangleq u_t+u\cdot\na u$ denotes the material derivative of
the velocity.
	Moreover, there exists a positive constant $\bar\alpha$ depending only on $\bar{\rho}$, and $\Omega$, such that, for  any $t \geq 1$ and  $ p\in[2,\infty)$,
	\begin{equation}\la{1500}
	\|\nabla u(\cdot,t)\|_{L^{p}}+\|P(\cdot,t)\|_{L^{p}}+\|\sqrt{\rho} \dot u(\cdot,t)\|_{L^2} +\|\nabla^2u(\cdot,t)\|_{L^2}+\|\na P(\cdot,t)\|_{L^2}\leq C e^{-\bar\alpha t}, 
	\end{equation}
	where   $C$ depends only on $p, \bar{\rho}, \|u_0\|_{L^2_\sigma}$, and $\Omega$.	
\end{theorem}

In particular, if the initial density is bounded away from vacuum, we have the following  uniqueness results of weak solutions.

\begin{theorem}\label{tm2}
  Let $T>0$ and $(\rho, u,  {P})$, $(\hat{\rho}, \hat{u},   \hat{{P}})$ (with $\int_\Omega \left(P,\hat P\right)dx=0$) be two weak solutions of \eqref{11}-\eqref{14} on $ \Om\times [0, T]$. Assume in addition that	
\begin{description}
\item	{\rm(i)}
$\sqrt{\rho}(u-\hat{u}) \in L^{\infty}\left(0, T ; L^2\right)$;\end{description}
\begin{description}
\item	{\rm(ii)}
 $\nabla\left(u-\hat{u} \right)   \in L^2\left(0, T ; L^2\right)$;\end{description}
\begin{description}
 \item	{\rm(iii)}
 $\rho_0$ is bounded away from zero and $(\rho, u),(\hat{\rho}, \hat{u})$ satisfy the energy inequality  \eqref{y}.\end{description}

Then $(\rho, u,   {P})=(\hat{\rho}, \hat{u},   \hat{{P}})$ on $\Om \times [0, T]$.\end{theorem}

\begin{remark}
Theorems \ref{tm1} and \ref{tm2} can be viewed as the global well-posedness result of weak solutions {  with density bounded away from zero} for 2D inhomogeneous {  Navier--Stokes} equations in  bounded  domains, with the natural energy level $u_0\in L^2_\sigma(\Omega)$. Compared with  {  Danchin and Mucha} \cite{DanchinMucha2019}, our  requirement of regularity on velocity weakens from $u_0\in H^1_{0,\sigma}(\Omega)$ to $u_0\in L^2_\sigma(\Omega)$, while requiring  a strictly positive lower bound on the initial density. Moreover, in contrast  to  \cite{wei_2025}, where addressed the whole space,  we extend the well-posedness results of weak solutions  to bounded domains, by applying  new techniques,  including   the Green representation,
the duality of BMO and  Hardy space, and the classical Sobolev extension theorem. 
\end{remark}

\begin{remark}
 In Theorem \ref{tm2},  positive lower bound on the initial density plays an essential role  in
our uniqueness argument. Specifically, it yields 
\[
\|u-\hat u\|_{L^2(\Om)}
\leq
\underline\rho^{-1/2}
\|\sqrt{\rho}(u-\hat u)\|_{L^2(\Om)},
\]
which is used to control the nonlinear terms in the energy inequality
for the velocity difference. 
 Moreover, it should be mentioned  that  a
recent result \cite{Violini2026} concerns the  whole-space
density-patch problem with $\rho_0=\mathbf 1_D$, and 
$u_0\in L^2_\sigma(\mathbb R^2)$. As for the bounded domains, with general 
 density   containing   vacuum, the well-posedness of weak solutions deserves further study  in our future work.
\end{remark}

We now comment on the analysis of this paper. 
To prove  Theorem \ref{tm1},  we need to establish a priori estimates
on smooth solutions in suitable norms.  Among these estimates, the key issue is to establish a {  uniform time-weighted} bound on the {  $L^2$}-norm of {  $\na u$} (see \eqref{313}). To this end,   testing the momentum equation \eqref{11}$_2$ by $\dot u$ (the material derivatives of the velocity) similar as   \cite{Li_2019, wei_2025, wei_patch_2024}, it turns out that the crucial point is to bound the pressure term
$$
I\triangleq \int_{\Omega} P\partial_i u^j\partial_j u^i\,dx.
$$
Motivated by \cite{Lu_2018},  we want to employ  the duality of BMO and $\mathcal{H}^1$ spaces, which is valid  for the whole space $\mathbb{R}^2$. So,   some new ideas  are needed to overcome the boundedness of domain $\Om$.  By extending $u$ simply by zero extension, and $P$ by the classical Sobolev extension (see  \cite[Section 5.4, Theorem 1]{Evans_2010}), we succeed in bounding $I$  by $\|\na P\|_{L^2}\|\nabla u\|_{L^{2}}^2$ (see Lemma \ref{lemmaP1}). Moreover, we find that  the $L^{2}$-norm of $\nabla u$ in fact decays exponentially for any $t\geq 1$, due to the boundedness of the domain $\Om$ (see \eqref{313}).

As for the proof of  Theorem \ref{tm2}, the key point is to derive the estimate on 
$t^{1/2}\nabla u\in L^2(0,T;L^\infty)$,  which essentially follows from  the following    point-wise
inequality for $\nabla u(x,t)$ (see \eqref{bdrylocal3})
\[
	|\na u(x_0, t)|
	\le
	Cr^{-1}\|\na u\|_{L^2(\Om)}
	+
	C\int_{E_r(x_0)}\frac{|\rho\dot u(y,t)|}{|x_0-y|}dy,\quad x_0\in \Om.
\]
To reach this point, for any $x_0\in\Omega$,   distinguishing     into  near-boundary point and interior point, we construct 
a smooth local domain $E_r(x_0)\subset\Omega$, respectively, with uniform normalized $C^2$ bounds
  in $x_0$ and $r$ (see Lemma \ref{lm:bdrylocal}). Moreover, on $E_r(x_0)$,  we decompose  problem  \eqref{11}  into a linear nonhomogeneous Dirichlet Stokes problem with source term $\rho\dot u$ and  a remaining homogeneous component. The desired estimate is then  followed  by combining   the gradient estimation of the  solutions to  the former, through  Green functions estimate and scaling and translation techniques, and  the   scale-uniform   \(C^1\) boundary  estimate for the latter. It is worth  noticing that  our method  differs substantially  from that of \cite{wei_patch_2024},  which relies on  the whole space, and is not  applicable to our present bounded domain.
   Then, motivated by \cite{wei_patch_2024}, we control the  principal  singular integral appearing  in the point-wise inequality for $\nabla u(x,t)$,  by  decomposing \(\dot u\) into its local average and the remainder. This,  coupled with the Poincaré and Lorentz-space estimates, leads to  $
t^{1/2}\nabla u\in L^2(0,T;L^\infty(\Omega))$ (see Lemma \ref{lm:bdrylemma23}).
Finally,  the uniqueness of Lions' weak solutions is obtained by a stability argument (see Lemma \ref{prop:stability-bdd}) and approximation by strong solutions.
Here, we should emphasize that, as in \cite{wei_2025}, the duality between \(W^{-1,4}(\Omega)\) and  \(W^{1,4/3}_0(\Omega)\) employed in the stability argument,  remains applicable to the  bounded  domains  with  Dirichlet boundary. 
   
\index{This weighted
$L^2_tL^\infty_x$ estimate permits the construction of a backward
transport solution with a controllable logarithmic loss. This yields`
\[
t^{-1}(\rho-\hat\rho)
 \in L^4(0,T;W^{-1,4}(\Omega)),
\]
which allows us to control the term involving \(\rho-\hat\rho\) in the energy inequality for \(u-\hat u\) through the duality between \(W^{-1,4}(\Omega)\) and  \(W^{1,4/3}_0(\Omega)\). Together with the energy estimate for $u-\hat{u}$ and the time-weighted bounds for $\dot{\hat{u}}$, this yields the stability
lemma (see Lemma \ref{prop:stability-bdd}). Finally,   both weak solutions are
then compared with the same sequence of strong solutions obtained by
 approximation in $H^1(\Om)$ of the initial velocity. The  stability
lemma forces the two weak solutions to have the same limit, which
yields uniqueness.}

The  paper is organized as follows. In Section \ref{2}, we list several known facts
  and some basic inequalities, which will be used in the subsequent analysis. Section \ref{3} is devoted to establishing some necessary a priori estimates.   Then, we prove Theorems \ref{tm1}-\ref{tm2} in Section \ref{4}. Finally, in Appendix \ref{sec:technical-lemmas}, we have provided a detailed proof of some  results used in Sections \ref{sec:bdry-linfty} and \ref{4.2}

\section{Preliminaries}\la{2}
In this section, we give some known facts and elementary inequalities which will be used later.

We start with the  following well-known Gagliardo-Nirenberg inequality.
\begin{lemma}[Gagliardo-Nirenberg]\la{lm21} Let $\Om\subset\mr^2$ be a smooth
	bounded domain. Then, for $q \in[2, \infty),\\ r \in(2, \infty)$, and $s \in(1, \infty)$, there exists a generic constant $C>0$ which  depends only on $q, r, s$ and $\Om$ such that for any $f \in H^{1}(\Omega)$ and $g \in L^{s}\left(\Omega\right) \cap W^{1, r}\left(\Omega\right)$, 
	\begin{equation*}
		\begin{aligned}
			&\|f\|_{L^{q}\left(\Omega\right)}^{q} \leq C\|f\|_{L^{2}\left(\Omega\right)}^{2}\|\nabla f\|_{L^{2}\left(\Omega\right)}^{q-2} + C\left\| f\right\|_{L^2(\Om)} ^{q}, \\
			&\|g\|_{C\left(\overline{\Omega}\right)} \leq C\|g\|_{L^{s}\left(\Omega\right)}^{s(r-2) /(2 r+s(r-2))}\|\nabla g\|_{L^{r}\left(\Omega\right)}^{2 r /(2 r+s(r-2))} +  C\left\| g\right\|_{L^2(\Om)}.
		\end{aligned}
		\label{22}
	\end{equation*}
 Moreover, 
\begin{equation*}
	\begin{aligned}
		&\|f-f_\Om\|_{L^{q}\left(\Omega\right)}^{q} \leq C\|f-f_\Om\|_{L^{2}\left(\Omega\right)}^{2}\|\nabla f\|_{L^{2}\left(\Omega\right)}^{q-2}.
	\end{aligned}
\end{equation*}
\end{lemma}

Let $\mathcal{H}^{1}(\mathbb{R}^2)$ and $\mathrm {BMO}(\mathbb{R}^{2})$ stand  for the usual Hardy and $\mathrm {BMO}$ spaces respectively (see \cite[Sections II.1--II.2]{FeffermanStein1972}). Here $\mathrm{BMO}$ is understood modulo constants, with the mean-oscillation norm over squares.
Then  the following well-known facts (see \cite[Section 2]{Clms} for (i)) play a key role in the  proof of Lemma \ref{lemmaP1} in the next section.
\begin{lemma}\label{lemma2.3}
	{\rm(i)} There is a positive constant $C$ such that
	\begin{equation*}
		\|E \cdot B\|_{\mathcal{H}^{1}\left(\mathbb{R}^{2}\right)} \leq C\|E\|_{L^{2}\left(\mathbb{R}^{2}\right)}\|B\|_{L^{2}\left(\mathbb{R}^{2}\right)},
		\label{23}
	\end{equation*}	
	for any $E \in L^{2}\left(\mathbb{R}^{2}\right)$ and $B \in L^{2}\left(\mathbb{R}^{2}\right)$ satisfying
	\begin{equation*}
		\operatorname{div} E=0,\quad \nabla^{\perp} \cdot B=0 \text { in } \mathcal{D}^{\prime}\left(\mathbb{R}^{2}\right) .
	\end{equation*}
	{\rm(ii)} There is a positive constant $C$ such that
	\begin{equation*}
		\|v\|_{\mathrm{BMO}\left(\mathbb{R}^{2}\right)} \leq C\|\nabla v\|_{L^{2}\left(\mathbb{R}^{2}\right)},
		\label{24}
	\end{equation*}
	for any $v \in \left\{v \in H_{\mathrm{loc}}^1(\mathbb{R}^2) \mid \nabla v \in L^2(\mathbb{R}^2)\right\}$.
\end{lemma}
{\it Proof of (ii). }For each square $Q\subset\mathbb R^2$, Poincar\'{e} inequality gives
\[
\begin{aligned}
\frac1{|Q|}\int_Q\left|v-\average{Q}v(y)\,dy\right|\,dx
&\le |Q|^{-1/2}\left\|v-\average{Q}v(y)\,dy\right\|_{L^2(Q)}\\
&\le C\|\nabla v\|_{L^2(Q)}
\le C\|\nabla v\|_{L^2(\mathbb R^2)}.
\end{aligned}
\]
Taking the supremum over all squares gives (ii).\hfill $\Box$

	Next,  we consider the divergence problem
\begin{equation}\la{div1}
	\operatorname{div} u=g,\,\, x\in\Om,\quad\left.\quad u\right|_{\partial \Omega}=0 .
\end{equation}
\begin{lemma}[{  {\cite[Theorem 4.9 and Remark 4.10]{CostabelMcIntosh2010}, \cite[Lemma 2.4]{peng_2021}}}]\la{div0}
	Let $\Om$ be  a bounded Lipschitz domain. Then  for every $g\in L^2(\Om)$ with $\int_\Om g dx=0$, there exists
	$u\in [W^{1,2}_0(\Om)]^2$ satisfying the problem \eqref{div1} and
\begin{equation*}\la{div2}
	\|\nabla u\|_{L^2(\Om)}\le C(\Om)\|g\|_{L^2(\Om)}.
\end{equation*}
\end{lemma}
%\noindent{\it Proof.} { Extend $g$ by zero to $\mathbb R^2$.
%By \cite[Theorem 4.9(c) and Remark 4.10]{CostabelMcIntosh2010}
%with $n=2$, $s=0$ and degree $2$, the zero-integral extension of $g$ admits
%$a_1,a_2\in H^1(\mathbb R^2)$, supported in $\overline\Omega$, such that}
%\begingroup
%\[
% \partial_1a_2-\partial_2a_1=g,\qquad
 %\|a_1\|_{H^1(\mathbb R^2)}+\|a_2\|_{H^1(\mathbb R^2)}
 %\le C(\Omega)\|g\|_{L^2(\Omega)}.
%\]
%\endgroup
%{ The support-space identification in \cite[Section 2]{CostabelMcIntosh2010}
%implies $a_1|_\Omega,a_2|_\Omega\in H^1_0(\Omega)$.
%Thus $u=(a_2,-a_1)|_\Omega$ satisfies \eqref{div1} and the asserted estimate.}
%\hfill { $\Box$}

  Let $L$ be a strongly elliptic operator of the form
\[
Lu=D_\alpha(A^{\alpha\beta}D_\beta u)
\]
 where the coefficients $A^{\alpha \beta}=A^{\alpha \beta}(x)$ are $2 \times 2$ matrix-valued functions on $\Omega$ with $A_{i j}^{\alpha \beta}$ satisfying the strong ellipticity condition, i.e., there is a constant $\lambda \in(0,1]$ such that
\begin{equation*}
	\left|A^{\alpha \beta}(x)\right| \leq \lambda^{-1}, \quad \sum_{\alpha, \beta=1}^2 A^{\alpha \beta}(x) \xi_\beta \cdot \xi_\alpha \geq \lambda \sum_{\alpha=1}^2\left|\xi_\alpha\right|^2 
\end{equation*}
for any $x \in \Omega$ and $\xi_\alpha, \xi_\beta\in \mathbb{R}^2, \alpha, \beta \in\{1,2\}$.  The adjoint operator ${L}^*$ is given by
$$
{L}^* u=D_\alpha\left(\left(A^{\beta \alpha}\right)^{\top} D_\beta u\right),
$$
where $\left(A^{\beta \alpha}\right)^{\top}$ is the transpose of the matrix $A^{\beta \alpha}$.

For the Stokes problem
	\begin{equation}\label{Stokes}
		\left\{\begin{aligned}
			{L}^* u+\nabla p&=f,\quad  x\in\Omega,\\
			\operatorname{div} u&=g,\quad x\in \Omega,  \\
			u&=0,  \quad x\in \partial \Omega,
		\end{aligned}\right.
	\end{equation}
	with $f\in [L^{\infty}(\Omega)]^2$ and $g\in L^{\infty}(\Omega), \int_\Om gdx=0$, we state the  following Green representation which will be useful for point-wise
estimate on $\nabla u(x,t)$. 
\begin{lemma}[{\cite[Theorem 3.2, Theorem 3.7, Remark 3.8]{ChoiKim_2019}}]\label{thm:CK37}
	Let $\Omega\subset\mathbb R^2$ has a $C^{1,\mathrm{Dini}}$ boundary. Suppose that the coefficients $A^{\alpha\beta}$ of $L$ are of
	Dini mean oscillation in $\Omega$. Then there exists a Green function 
	$(G,\Pi)$ of $L$ and $(G^*,\Pi^*)$ of $L^*$ such that  if $(u,p)\in [W^{1,2}_0(\Omega)]^2\times L^2(\Omega)$ is a weak solution
	of the Stokes problem \eqref{Stokes}, then  
	\begin{equation*}
		u(y)=-\int_{\Omega} G(x, y)^{\top} f(x) d x+\int_{\Omega} \Pi(x, y)^{\top} g(x) d x,\quad \text{for a.e. $y\in\Om$,}
	\end{equation*}
where $G = G(x, y)$ is a   $2\times2$ matrix-valued function and
  $\Pi=\Pi(x,y)$ is a   $1\times2$ vector-valued   function,  normalized by
\begingroup 
\[
 LG(\cdot,y)+\nabla\Pi(\cdot,y)=-\delta_y I_2,\qquad
 \operatorname{div}G(\cdot,y)=0,\qquad
 G(\cdot,y)|_{\partial\Omega}=0,\qquad
 \int_\Omega\Pi(x,y)\,dx=0.
\]
\endgroup
Moreover, the following properties hold.
\begin{description}
	\item	{\rm(i)} $G$ and $G^*$ are continuous in
	\(	\{(x,y)\in\Omega\times\Omega:\ x\neq y\}\)
 and satisfy
	\begingroup
\begin{equation}\label{CK32}
		G(x,y)=G^*(y,x)^\top,\quad \text{for any } x, y\in\Om,\quad x\neq y;
	\end{equation}
\endgroup
\end{description}
\begin{description}
	\item	{\rm(ii)} for any $y\in\Omega$ and $R>0$,
	\begingroup
\begin{equation}\label{CK38}
		(G(\cdot,y),\Pi(\cdot,y))
		\in
		C^1(\Omega\setminus B_R(y))^{2\times 2}\times C(\Omega\setminus B_R(y))^2;
	\end{equation}
\endgroup
\end{description}
\begin{description}
	\item	{\rm(iii)} for any $x,y\in\Omega$ with $x\neq y$,
	\begingroup
\begin{equation*}%\label{CK39}
		|D_xG(x,y)|+|\Pi(x,y)|
		\le
		C|x-y|^{-1},
	\end{equation*}
\endgroup
   where the constant $C$ depends only on $\lambda$,  $\operatorname{diam}(\Omega)$, the
Dini modulus of the coefficients, and the $C^{1,\mathrm{Dini}}$ character of
$\partial\Omega$.
%{ Here this boundary character includes a common graph radius $R_*>0$
%and a modulus $\varrho_*$ for the graph derivatives satisfying
%$\int_0^1\varrho_*(s)\,ds/s<\infty$.}
 \end{description}
\end{lemma}
\begin{remark}
 Indeed,  Lemma \ref{thm:CK37} holds  for $C^2$ domains since  they possess $C^{1, \mathrm{Dini}}$ boundary. \end{remark}

\begin{lemma}[{\cite[Theorem IV.5.1]{G_1998}}]
	\label{thm:Galdi-local-boundary-Stokes}
	Let $\Omega\subset\mathbb R^n$, $n\ge2$, be an arbitrary domain with a boundary portion
	$\sigma$ of class $C^{m+2}$, $m\ge0$ and  $\Omega^{0}$ be any bounded subdomain of $\Omega$
	with
	\[
	\partial\Omega^{0}\cap\partial\Omega=\sigma.
	\]
Moreover, let
	\[
	u\in W^{1,q}(\Omega^{0}),\quad P\in L^{q}(\Omega^{0}),\quad 1<q<\infty,
	\]
	be such that
	\[
	(\nabla u,\nabla\phi)
	=
	-\langle f,\phi\rangle
	+
	(P,\nabla\cdot\phi), 
	\quad
	\forall\,\phi\in C_{0}^{\infty}(\Omega^{0}),
	\]
	\[
	(u,\nabla\phi)=0, 
	\qquad
	\forall\,\phi\in C_{0}^{\infty}(\Omega^{0}),
	\]
	and
	\[
	u=u^{*}\quad\text{on }\sigma .
	\]
	
	Then, if
	\[
	f\in W^{m,q}(\Omega^{0}),
	\quad
	u^{*}\in W^{m+2-\frac1q,q}(\sigma),
	\]
we have 
	\[
	u\in W^{m+2,q}(\Omega_{0}),
	\quad
	P\in W^{m+1,q}(\Omega_{0}),
	\]
	for any $\Omega_{0}$ satisfying

\begin{description}
	\item	{\rm(i)}  $\Omega_{0}\subset\Omega^{0}$;
    \end{description}
\begin{description}\item{\rm(ii)} {  $\partial\Omega_0\cap\partial\Omega$} is a strictly interior subregion of $\sigma$.
\end{description}
\begin{description}\item{\rm(iii)}
{  $\operatorname{dist}(\overline{\Omega_0},
 \overline{\partial\Omega^0\setminus\partial\Omega})>0$.}
\end{description}
	
Furthermore, there exists  positive constant 	$C=C(m,n,q,\Omega^{0},\Omega_{0})$ such that
	\begin{equation*}%\label{Galdi-IV-5-23}
		\|u\|_{W^{m+2,q}(\Omega_{0})}
		+
		\|P\|_{W^{m+1,q}(\Omega_{0})}
		\le
		C\Big(
		\|f\|_{W^{m,q}(\Omega^{0})}
		+
		\|u^{*}\|_{W^{m+2-1/q,q}(\sigma)}
		+
		\|u\|_{W^{1,q}(\Omega^{0})}
		+
		\|P\|_{L^{q}(\Omega^{0})}
\Big).
\end{equation*}
\end{lemma}

Moreover, we state the standard interior \(W^{2,p}\) estimate for
homogeneous Stokes systems (see \cite[Lemma 2.12]{Tsai2018}).

\begin{lemma}
	\label{lm:stokes-int-w2p-lp}
	Let \(1<q<\infty\). Suppose that
	\[
	u\in W^{1,1}_{\rm loc}(B(0,1))^2,\quad
	P\in L^1_{\rm loc}(B(0,1))
	\]
	satisfy 
    \[
	(\nabla u,\nabla\phi)
	=
	(P,\nabla\cdot\phi), 
	\quad
	\forall\,\phi\in C_{0}^{\infty}(B(0,1)),
	\] and \[
	(u,\nabla\phi)=0, 
	\quad
	\forall\,\phi\in C_{0}^{\infty}(B(0,1)).
	\]
	Assume also that \(u\in L^q(B(0,1))^2\),  then
	\[
	u\in W^{2,q}(B(0,{1/2}))^2,\quad \nabla P\in L^q(B(0,{1/2}))^2,
	\]
	and
\begin{equation*}%\label{stokes-int-w2p-lp-est}
		\|u\|_{W^{2,q}(B(0,{1/2}))}
		+
		\|\nabla P\|_{L^q(B(0,{1/2}))}
\le
C(q)\|u\|_{L^q(B(0,1))}.
\end{equation*}
\end{lemma}
{\it Proof.} Applying \cite[Lemma 2.12, (2.67) and (2.69)]{Tsai2018}
with zero forcing and radii $1/2$ and $1$, we obtain
\[
\begin{aligned}
\|\nabla u\|_{L^q(B(0,1/2))}+\|\nabla^2u\|_{L^q(B(0,1/2))}
&\le C(q)\|u\|_{L^1(B(0,1))}\\
&\le C(q)\|u\|_{L^q(B(0,1))}.
\end{aligned}
\]
Since $\nabla P=\Delta u$ in distributions,
\[
\|\nabla P\|_{L^q(B(0,1/2))}
\le C\|\nabla^2u\|_{L^q(B(0,1/2))}.
\]
Together with $\|u\|_{L^q(B(0,1/2))}\le\|u\|_{L^q(B(0,1))}$,
this proves Lemma \ref{lm:stokes-int-w2p-lp}.\hfill $\Box$

Arguing as in the proof of   \cite[Lemma 5.1]{wei_2025}, we 
obtain the following    backward transport result in  bounded domain. 
\begin{lemma}\label{lem:transport-bdd}
	Let $\Om\subset\mr^2$ be a bounded {  smooth} domain, $T>0$ and   $1<p<\infty$.  Assume that
	\begin{equation*}\la{nabaru}
		s^{1/2}\nabla  u\in L^2\bigl(0, T;L^\infty(\Omega)\bigr),
		\quad
		\div  u=0,
		\quad
	 u|_{\partial\Omega}=0.
	\end{equation*}
	Then, for fixed $0<t\leq T$ and any $\varphi\in W^{1,p}_0(\Omega)$, the  problem
\begin{equation*}\label{eq:back-transport}
		\left\{
		\begin{aligned}
			&\partial_s\phi+ u\cdot\nabla\phi=0,
			&& (x,s)\in\Omega\times(0,t],\\
			&\phi(t,x)=\varphi(x),&& x\in\Om,
		\end{aligned}
		\right.
	\end{equation*}
	admits a unique solution
	\[
	\phi\in  C\bigl((0,t];W^{1,p}_0(\Omega)\bigr),
	\]
	and there exists a constant $C=C(p)$ such that for any $s\in(0,t]$,
	\begin{equation*}\label{eq:transport-est-2}
		\|\nabla\phi(\cdot,s)\|_{L^p(\Omega)}
		\le
		C\exp\Bigl(
		|\ln(t/s)|^{1/2}
		\|\tau^{1/2}\nabla   u\|_{L^2(0,t;L^\infty(\Omega))}
		\Bigr)\|\nabla\varphi\|_{L^p(\Omega)}.
	\end{equation*}
\end{lemma}

Finally, we also need the    one-dimensional convolution estimate. 
\begin{lemma}[{  {\cite[Lemma 5.2]{wei_2025}}}]\label{lem:kernel}
	Let $T>0$ and $f:(0,T]\rightarrow [0, +\infty]$ be such that $f\in L^4(0,T).$ Let $A>0$ and 
	\[
	F(t)\triangleq\frac1t \int_0^t f(s)e^{A|\log(t/s)|^{1/2}}\,ds,
	\quad   0<t\leq T.
	\]
	Then $F\in L^4(0,T)$ and
\[
	\|F\|_{L^4(0,T)}\le C(A) \|f\|_{L^4(0,T)},
	\]
where  $C(A)>0$  is a constant depending only on $A$.
\end{lemma}

\section{A priori estimates}\la{3}

Let $T>0$ be fixed. In this section, we establish necessary a priori estimates for smooth solutions $(\rho, u, P)$ to the Dirichlet problem \eqref{11}{  --}\eqref{14} on $\Omega \times [0,T]$ with initial data $(\rho_0, u_0)$ satisfying \eqref{15}.
{ Throughout this section, we choose $P_\Om(t)=0$. Subtracting the spatial mean does not change \eqref{11}; full pressure norms below refer to this representative.}

\subsection{Lower order estimates
}

First,  the transport equation $\eqref{11}_1$  and $\eqref{11}_3$ give the following well-known estimate of density.

\begin{lemma}\label{Lemma 31}{\rm(\cite[Theorem 2.1]{Lions_1996})} Let $(\rho,u,P)$ be a smooth solution to \eqref{11}-\eqref{14} satisfying \eqref{15}. It holds that for any $(x, t) \in \Omega \times(0, T)$
	\begin{equation}\label{31}
		0 \leq \rho(x, t) \leq \bar{\rho}.
	\end{equation}
\end{lemma}

The following  estimate is  essential to  ensure  the validity of the results in  $\mathcal{H}^{1}(\mathbb{R}^2)$ and $\mathrm {BMO}(\mathbb{R}^{2})$, when restricted to bounded domains.
\begin{lemma}\label{lemmaP1} Let $(\rho,u,P)$ be a smooth solution to \eqref{11}-\eqref{14} on $\Om\times[0,T]$ satisfying \eqref{15}. Then there exists a positive constant $C$ depending only on $ \Omega$, such that
	\begin{equation*}
		|\int_{\Omega} P\partial_i u^j \partial_j u^i dx |\leq C\|\na P\|_{L^2}\|\nabla u\|^2_{L^2}.
		\label{37}
	\end{equation*}
\end{lemma}
{\it Proof. }We denote by \(\tilde u\) the zero extension of \(u\) to \(\mathbb R^2\),  and $\tilde{P}$  the general Sobolev extension  of $P$ to $\mr^2$ (see \cite[Theorem 1, Section 5.4]{Evans_2010}),  which satisfies
\begin{equation}
	\|\tilde{P}\|_{H^{1}\left(\mathbb{R}^{2}\right)} \leq C\left\|P\right\|_{H^{1}(\Omega)}\le C\|\nabla P\|_{L^2(\Om)}.
	\label{38}
\end{equation}
Since $u|_{\partial\Omega}=0,$ it is standard to check that
\begin{equation*}
	\partial_i \tilde{u} =
	\left\{\begin{array}{l}
		\partial_i u, \quad x\in \Omega,\\
		0, \qquad  x\in \mathbb{R}^2\setminus \Omega.
	\end{array}\right.
\end{equation*}
Then  for any $\varphi \in D (\mathbb{R}^2)$, we have
\begin{equation*}
	-\int_{\mathbb{R}^{2}} \partial_{j} \tilde{u}^{i} \partial_{i} \varphi dx
	= \int_{\mathbb{R}^{2}} \tilde{u}^{i} \partial_{ij} \varphi dx= \int_{\Omega}  {u}^{i} \partial_{ij} \varphi dx=- \int_{\Omega} \partial_{i } {u}^{i} \partial_{ j} \varphi dx
	=0,
\end{equation*}
which implies
$$\mathrm{div} \partial_{j} \tilde{u} =0 \mbox{ in } D'(\mathbb{R}^2).
$$
Similarly, it holds that
$$ \nabla^{\perp} \cdot\left(\nabla \tilde{u}^{j}\right) =0 \mbox{ in } D'(\mathbb{R}^2).$$

Thus,  it follows from Lemma \ref{lemma2.3}, \eqref{38}, the duality of $\mathcal{H}^1$ and $\mathrm{BMO}$ \cite[Theorem 2]{FeffermanStein1972}, and Poincar\'{e} inequality  that
\begin{equation*}
	\begin{aligned}	
		|\int_{\Omega} P \partial_i u^j \partial_j u^i dx|& = |\int_{\mathbb{R}^2} \tilde{P} \partial_i {\tilde u}^j \partial_j \tilde{u}^i dx|\\&\leq C\|\tilde{P}\|_{\mathrm{BMO}(\mathbb{R}^2)}\left\|\partial_{j} \tilde{u}^{i} \partial_{i} \tilde{u}^{j}\right\|_{\mathcal{H}^{1}(\mathbb{R}^2)} \\
		&\leq C\|\nabla \tilde{P}\|_{L^2(\mathbb{R}^2)}\|\nabla \tilde{u}\|_{L^{2}\left(\mathbb{R}^{2}\right)}^{2}\\
		&\leq C\|\na P\|_{L^2(\Omega)} \|\nabla u\|_{L^{2}(\Omega)}^2,
	\end{aligned}
\end{equation*}
which indicates  Lemma \ref{lemmaP1}.\hfill $\Box$

Next, we shall give  the following  exponential decay rate in time  for $\|\sqrt{\rho}u\|_{L^2}$.
\begin{lemma} \label{lemma 3.4}Let $(\rho,u,P)$ be a smooth solution to \eqref{11}-\eqref{14} on $\Om\times[0,T]$ satisfying \eqref{15}.  Then there exists a positive constant $C$ depending on $ \bar\rho$, $\|u_0\|_{L^2_\sigma}$, and $\Omega$ such that
	\begin{equation}\label{33}
		\sup_{t \in[0, T]} e^{\alpha t} \|\sqrt\rho u\|_{L^2}^{2} + \int_0^T e^{\alpha t}\|\nabla u\|_{L^2}^{2} dt\leq C,
	\end{equation}
	with some positive constant $\alpha$ depending only on $\bar{\rho}$, and $\Omega$.
\end{lemma} {\it Proof. } Multiplying $\eqref{11}_2$ by $u$ and integrating by parts lead to
\begin{equation}
	\frac{1}{2} \frac{d}{d t} \int_{\Omega} \rho|u|^{2} d x+ \int_{\Omega} |\nabla u|^{2} d x=0.
	\label{340}
\end{equation}

By virtue of  Poincar\'{e} inequality and \eqref{31},  we have
\begin{equation}\label{35} 
	\int_{\Omega} \rho|u|^{2} dx \leq \bar{\rho}\|u\|_{L^{2}}^{2}  \leq \bar{\rho}C(\Omega)\|\nabla u\|_{L^{2}}^{2}\leq \frac{1}{2\alpha}\|\nabla u\|_{L^{2}}^{2},
\end{equation}
where $\alpha$ is a positive constant depending only on $\bar{\rho}$ and $\Omega$. This and \eqref{340} show that
\begin{equation}
	\frac{d}{d t} \int_{\Omega} \rho|u|^{2} d x+ \alpha\int_{\Omega} \rho|u|^{2} d x+ \int_{\Omega} |\nabla u|^{2} d x\leq0.
	\label{34}
\end{equation}

Then, multiplying \eqref{34} by $e^{\alpha t}$ and combining with \eqref{35}, we get
\begin{equation*}
	\sup_{t \in[0, T]} e^{\alpha t} \|\sqrt\rho u\|_{L^2}^{2} + \int_0^T e^{\alpha t}\left(\int_{\Om}|\nabla u|^{2}dx\right) dt\leq C,
\end{equation*}
which completes the proof
   of Lemma \ref{lemma 3.4}.\hfill $\Box$

With Lemmas \ref{Lemma 31}-\ref{lemma 3.4} at hand, we establish  the  time-exponential decay estimate for $\|\na u\|_{L^2}$.
\begin{lemma} \label{lemma 3.40} Let $(\rho,u,P)$ be a smooth solution to \eqref{11}-\eqref{14} on $\Om\times[0,T]$  satisfying \eqref{15}. Then  there exists a positive constant $C$ depending on $ \bar\rho$, $\|u_0\|_{L^2_\sigma}$, and $\Omega$ such that
	\begin{equation}
		\sup _{t \in[0, T]} \eta e^{\alpha t}\|\nabla u\|_{L^{2}}^{2}+\int_{0}^{T}\eta e^{\alpha t} \|\sqrt{\rho}\dot{u}\|_{L^2}^2dt \leq C,
		\label{313}
	\end{equation}
	where $\alpha$ is  the same as Lemma \ref{lemma 3.4} and  $\eta(t)\triangleq\min\{1,t\}$.
\end{lemma} {\it Proof. }First, the standard $L^{p}$-estimate to \eqref{11} (see \cite[Theorem 4.6, $m=2$]{AmroucheGirault1994})  gives 
	\begin{equation}
\left\| \na^2u\right\|_{L^{r}} +\left\|\nabla P\right\|_{L^ r} \leq C\left\|\rho \dot{u}\right\|_{L^r},\quad \forall r\in (1,\infty),
\label{315}
\end{equation}
which together with
\[\int_\Om  \na udx=0, \]
and  Gagliardo-Nirenberg inequality implies that 
\begin{equation}
\|\nabla u\|_{L^p}
\leq C\|\nabla u\|_{L^2}^{\frac2p} \|\rho \dot{u}\|_{L^2}^{1-\frac2p},\quad \forall p\in (2,\infty).	
\label{418}
\end{equation} 

Since $u=0$ on $\partial\Om$, differentiation in time gives $u_t=0$ there, and $u\cdot\nabla u=0$ there as well. Thus $\dot u=0$ on $\partial\Om$. Next,  motivated by \cite{Lu_2018},  integrating  $\eqref{11}_{2}$ multiplied  by $\dot{u}$  over $\Omega$ yields
\begin{equation}
	\int_{\Omega} \rho|\dot{u}|^{2} dx=\int_{\Omega}\left(\Delta u \cdot \dot{u}-\nabla P \cdot \dot{u} \right)dx \triangleq I_{1}+I_{2} .
	\label{316}
\end{equation}

Now,  we estimate $I_1, I_2$ as follows.

By \eqref{11}$_3$,   \eqref{418} and Cauchy inequality, we deduce 
\begin{equation}
	\begin{aligned}
		I_{1}=&-\frac12\frac{d}{dt}\int_{\Omega} |\nabla u|^{2}dx-\int_{\Omega} \partial_{i} u^{j}\partial_{i}(u^k \partial_{k}u^j)dx\\=&-\frac12\frac{d}{dt}\int_{\Omega} |\nabla u|^{2}dx-\int_{\Omega} \partial_{i} u^{j} \partial_{i}u^k \partial_{k}u^jdx-\frac12\int_{\Omega}\partial_k(u^k|\partial_iu^j|^2)dx\\
		\leq& -\frac12\frac{d}{d t} \int_{\Omega} |\nabla u|^{2} dx+C\|\nabla u\|_{L^{2}}\|\nabla u\|_{L^{4}}^{2}\\
		\leq &-\frac12\frac{d}{d t} \int_{\Omega}|\nabla u|^{2} dx +\frac14\|\sqrt{\rho} \dot{u}\|_{L^{2}}^{2}+C \left \| \nabla u\right \|_{L^{2}}^{4}.
	\end{aligned}
	\label{317}
\end{equation}

For $I_2$, it is standard to check from $\eqref{11}_3$ and  integration by parts 
\begin{equation*}
	\begin{aligned}	
		I_{2}=-\int_\Om 
		\partial_i P(\partial_t u^i+u^j\partial_j u^i )dx=\int_\Omega P\partial_ju^i\partial_iu^jdx.
	\end{aligned}
\end{equation*}
Moreover,  by Lemma \ref{lemmaP1} and \eqref{315}, we have 
\begin{equation}
	\begin{aligned}	
		I_{2}\leq C\|\nabla P\|_{L^2}\|\nabla u\|_{L^2}^2
		\leq C\|\rho \dot{u}\|_{L^2}\|\nabla u\|_{L^2}^2\leq \frac{1}{4}\|\sqrt{\rho} \dot{u}\|_{L^2}^2+C\|\nabla u\|_{L^2}^4.
	\end{aligned}
	\label{318}
\end{equation}

Substituting  \eqref{317}, \eqref{318} into \eqref{316} gives
\begin{equation}\la{nau}
	\frac{d}{dt}\int_{\Om}|\nabla u|^2dx+\int_{\Om}\rho |\dot{u}|^2 dx\leq C\|\nabla u\|_{L^2}^4.
\end{equation}
Finally, multiplying  
\eqref{nau}  by $\eta e^{\alpha t}$, one has
\begin{equation*}
	\frac{d}{dt}\biggl(\eta e^{\alpha t}\int_{\Om}|\nabla u|^2dx\biggr)+\eta e^{\alpha t}\int_{\Om}\rho |\dot{u}|^2 dx\leq C\eta e^{\alpha t}\|\nabla u\|_{L^2}^4+Ce^{\alpha t}\|\nabla u\|_{L^2}^2,
\end{equation*}
which, along with \eqref{33} and  Gronwall's inequality,  leads to  \eqref{313}.  Thus,  the proof of Lemma \ref{lemma 3.40} is completed.  \hfill $\Box$

Now, we are in a position to derive the following estimate on the material derivative
of velocity.
\begin{lemma}\label{lm32}
	Let $(\rho,u,P)$ be a smooth solution to \eqref{11}-\eqref{14} on $\Om\times[0,T]$ satisfying \eqref{15}. 	Then there exists a positive constant $C$ depending only on  $\bar\rho$, $\left\|u_0\right\|_{L_\sigma^2}$,    and $\Omega$  such that 
	\begin{equation}
		\sup _{t \in[0, T]} \eta^2 e^{\alpha t}\left\|\sqrt{\rho} \dot{u}\right\|_{L^{2}}^{2}+\int_{0}^{T} \eta^2 e^{\alpha t} \left\|\nabla \dot{u}\right\|_{L^{2}}^{2} d t \leq C,
		\label{32600}
	\end{equation}
	where $\alpha$, $\eta$  are  the same as Lemma \ref{lemma 3.40}.
\end{lemma}
{\it Proof.} Operating $\dot{u}^{j}(\partial_t+u\cdot \nabla)$ to $\eqref{11}_2^j$, and integrating the resulting equation over $\Om$, one has after integration by parts  
\begin{equation}\la{j}
	\begin{aligned}
		\frac{1}{2} \frac{d}{d t} \int_\Omega \rho|\dot{u}|^2 d x+ \int_\Omega|\nabla \dot{u}|^2 d x =&-\int_\Omega  \partial_i\left(\partial_i u \cdot \nabla u^j\right) \dot{u}^j dx-\int_\Omega \operatorname{div}\left(\partial_i u \partial_i u^j\right) \dot{u}^j dx\\&-\int_\Omega \dot{u}^j\partial_t\partial_jPdx-\int_\Omega \dot{u}^ju\cdot\nabla\partial_jPdx\\\triangleq&\sum_{i=1}^{4}J_{i}.
	\end{aligned}
\end{equation}

We will estimate each term on
the right-hand side of \eqref{j} as follows. 

First, direct calculations show
\begin{equation}\la{j1j2}
	\begin{aligned}
		J_1+J_2&=\int_\Om(\partial_iu\cdot\na u^j)\partial_i\dot u^j dx+\int_\Om (\partial_iu\partial_iu^j)\cdot\na\dot{u}^jdx\\&\leq\frac{1}{4}\|\na\dot u\|_{L^2}^2+C\|\na u\|_{L^4}^4.
	\end{aligned}
\end{equation}

Next,  it follows from $\dot u|_{\partial\Omega}=0,$ \eqref{11}$_3$ and the fact $$(\partial_t+u\cdot\nabla)\partial_iP=\partial_i \dot P-\partial_i u\cdot\nabla P$$ that 
\begin{equation*}\la{j3j4}
	\begin{aligned}
		J_3+J_4=&\int_\Om \partial_t P\operatorname{div}\dot{u} dx+\int_\Om (u\cdot\na P)\operatorname{div}\dot u dx+\int_\Om(\partial_{ j}u\cdot\na P)\dot{u}^j dx\\=&\frac{d}{d t} \int_\Omega P \partial_{j} u^{k} \partial_{k} u^{j} dx-\int_\Omega P \left(\partial_{t}+u\cdot\nabla\right)\left( \partial_j u^k \partial_k {u}^j\right)dx- \int_\Omega  P \partial_j u^k \partial_k\dot{u}^j   dx \\
		=&\frac{d}{d t} \int_\Omega P \partial_{j} u^{k} \partial_{k} u^{j} dx  +\int_\Omega P (\partial_{i} u^{j} \partial_{j} u^{k} \partial_{k} u^{i}+\partial_{i} u^{k} \partial_{k} u^{j} \partial_{j} u^{i}) dx\\&-3 \int_\Omega P \partial_{i} u^{k}  \partial_{k}\dot{u}^{i} dx.
        \\\triangleq&\frac{d}{d t}\int_\Omega P \partial_{j} u^{k} \partial_{k} u^{j} dx +\sum_{i=5}^{6}J_{i}.
	\end{aligned}
\end{equation*}
We estimate  $J_i(i = 5,6)$  term by term.
   
 First,  by   $\operatorname{div}u=0$, we know   that  the $2\times2$ matrix $A\triangleq\na u$ has trace zero. This and the Cayley-Hamilton
identity give     
$$A^2=-\operatorname{det}(A) I,$$
which, together with $\operatorname{div}u=0$,  leads to      
$$\partial_i u^j \partial_j u^k \partial_k u^i=\partial_i u^k \partial_k u^j \partial_j u^i=\operatorname{tr}\left(A^3\right)=-\operatorname{det}A\operatorname{tr}(A)=0.$$
Then
 \begin{equation*}
	\begin{aligned}\la{j42}
	J_5=\int_\Omega P \partial_{i} u^{j} \partial_{j} u^{k} \partial_{k} u^{i} dx+\int_\Omega P \partial_{i} u^{k} \partial_{k} u^{j} \partial_{j} u^{i} dx=0. 
	\end{aligned}
\end{equation*}

 Next, we set  $\tilde{\dot{u}}$ as the zero extension of $\dot{u}$ to $\mr^2$. 
 
 Noticing that    $\dot u|_{\partial\Omega}=0$ gives
\begin{equation*}
	\partial_i \tilde{\dot u} =
	\left\{\begin{array}{l}
		\partial_i \dot u, \quad x\in \Omega,\\
		0, \quad \quad x\in \mathbb{R}^2\setminus \Omega.
	\end{array}\right.
\end{equation*}
So, we find
\begin{equation}\la{j421}
    \|\nabla\tilde{\dot{ u}}\|_{L^2(\mathbb R^2)}
=\|\nabla\dot u\|_{L^2(\Omega)}.
\end{equation}

Set 
\[
\dot{\tilde u}
\triangleq\partial_t\tilde u+\tilde u\cdot\nabla\tilde u.
\]
Since $u|_{\partial\Omega}=0$, we can check that 
\[\dot{\tilde u}=\tilde{\dot{u}}, \quad\text{a.e. in }\mr^2,\]
which, combined with \eqref{j421}, implies \[\na \dot{\tilde u}=\na \tilde{\dot{u}}, \quad\text{a.e. in }\mr^2.\]
Therefore, 
\begin{equation*}
	\begin{aligned}
\int_{\mr^2}  \tilde P \partial_{i} \tilde u^{k} \partial_{k} {\dot{\tilde u}}^{i} dx=\int_{\mr^2} \tilde P \partial_{i} \tilde u^{k} \partial_{k} \tilde {\dot{u}}^{i} dx.
	\end{aligned}
\end{equation*}

Combining this with     \eqref{38}, \eqref{j421}, using  Lemma \ref{lemma2.3},   and noticing that \[ \mathrm{div} \partial_{i} \tilde{u} =0 \mbox{ in } D'(\mathbb{R}^2),\quad\nabla^{\perp} \cdot\left(\nabla \tilde{\dot{u}}^{i} \right) =0 \mbox{ in } D'(\mathbb{R}^2),\] 
we get
\begin{equation}
	\begin{aligned}
		|J_6|&\leq C\|\tilde{P}\|_{\mathrm{BMO}(\mathbb{R}^2)}\left\|\partial_{i} \tilde{u}^{k} \partial_{k} \tilde{\dot u}^{i}\right\|_{\mathcal{H}^{1}(\mathbb{R}^2)} \\
		&\leq C\|\nabla \tilde{P}\|_{L^2(\mathbb{R}^2)}\|\nabla \tilde{u}\|_{L^{2}\left(\mathbb{R}^{2}\right)}\|\nabla \tilde{\dot u}\|_{L^{2}\left(\mathbb{R}^{2}\right)}\\
		&\leq C\|\na P\|_{L^2(\Omega)} \|\nabla u\|_{L^{2}(\Omega)}\|\nabla \dot u\|_{L^{2}(\Omega)}\\&\leq \frac{1}{4}\|\na\dot u\|_{L^2}^2+C\|\na P\|_{L^{2}}^{2}\|\na u\|_{L^{2}}^{2}.\la{j41}
	\end{aligned}
\end{equation}

Then, substituting    \eqref{j1j2}-\eqref{j41} into \eqref{j}, it holds that
\begin{equation}
	\Psi^{\prime}(t)+\frac{1}{2}\|\nabla \dot{u} \|_{L^{2}}^2\leq C\|\na P\|_{L^{2}}^{2}\|\na u\|_{L^{2}}^{2}+C\|\nabla u\|_{L^{4}}^{4},
	\label{324}
\end{equation}
where
\begin{equation*}
	\Psi(t) \triangleq \frac{1}{2} \int_\Omega \rho|\dot{u}|^{2} dx-\int_\Omega P \partial_{j} u^{k} \partial_{k} u^{j} dx
\end{equation*}
satisfies
\begin{equation}\label{326}
	\frac{1}{4} \int_\Omega \rho|\dot{u}|^{2} dx - C\|\nabla u\|_{L^2}^4 \leq \Psi(t) \leq \int_\Omega \rho|\dot{u}|^{2} dx + C\|\nabla u\|_{L^2}^4,
\end{equation}
due to \eqref{313} and \eqref{318}.

Moreover, it follows  from \eqref{324},  \eqref{315}, and \eqref{418}  that 
\begin{equation}\label{327}
	\Psi^{\prime}(t)+\frac{1}{2}\|\nabla \dot{u} \|_{L^{2}}^2\leq C\|\sqrt{\rho}\dot{u} \|_{L^{2}}^{2}\|\nabla u\|_{L^2}^2.
\end{equation}

Finally,  we derive \eqref{32600}  by  \eqref{33},  \eqref{313}, \eqref{326}, \eqref{327}, and Gronwall's inequality. \hfill $\Box$

\begin{lemma}\label{lm33}
	Let $(\rho,u,P)$ be a smooth solution to \eqref{11}-\eqref{14} on $\Om\times[0,T]$ satisfying \eqref{15}. 	Then for any $p\in(2,\infty)$, there exists a positive constant $C$ depending only on $p$,  $\bar\rho$, $\left\|u_0\right\|_{L_\sigma^2}$,    and $\Omega$  such that 
	\begin{equation}
		\sup _{t \in[0, T]} \eta^2 e^{\alpha t}\left(\left\|\na u\right\|_{L^p}^2+\left\|P\right\|_{L^p}^2\right)+ \int_{0}^T\eta e^{\alpha t}\left(\|\na u\|_{L^p}^2+\|P\|_{L^p}^2\right)dt\leq C,
		\label{lm36a}
	\end{equation}
	\begin{equation}
		\begin{aligned}
				\sup_{t \in[0, T]}\eta^2e^{\alpha t}\left(\|\na^2u\|_{L^2}^2+\|\na P\|_{L^2}^2\right)+\int_{0}^T\eta e^{\alpha t}\left(\|\na^2u\|_{L^2}^2+\|\na P\|_{L^2}^2\right)dt\leq C,
		\end{aligned}
	\label{lm36aa}
\end{equation}
and	
\begin{equation}
\int_{0}^T\eta^2 e^{\alpha t}\left(\|\na^2u\|_{L^p}^2+\|\na P\|_{L^p}^2\right)dt\leq C,
	\label{lm36aaa}
\end{equation}	
where $\alpha$, $\eta$  are  the same as Lemma \ref{lemma 3.40}.
\end{lemma}
{\it Proof.} By \eqref{315}, \eqref{418},  and Poincar\'{e} inequality, one has for any $p\in(2,\infty)$
\begin{equation*}
	\|\na u\|_{L^p}+\|P\|_{L^p}\leq C\|\na P\|_{L^2}+C\|\sqrt{\rho}\dot u\|_{L^2}+C\|\na u\|_{L^2}\leq C\|\sqrt{\rho}\dot u\|_{L^2}+C\|\na u\|_{L^2},
\end{equation*}
which, together with \eqref{313} and \eqref{32600}, yields that 
\begin{equation}
	\sup_{t \in[0, T]}\eta^2e^{\alpha t}\left(\|\na u\|_{L^p}^2+\|P\|_{L^p}^2\right)+\int_{0}^T\eta e^{\alpha t}\left(\|\na u\|_{L^p}^2+\|P\|_{L^p}^2\right)dt\leq C.\la{lm36b}
\end{equation}

Next, it follows from \eqref{313} and \eqref{32600} that 
\begin{equation}
	\begin{aligned}
		\sup_{t \in[0, T]}&\eta^2e^{\alpha t}\left(\|\na^2u\|_{L^2}^2+\|\na P\|_{L^2}^2\right)+\int_{0}^T\eta e^{\alpha t}\left(\|\na^2u\|_{L^2}^2+\|\na P\|_{L^2}^2\right)dt\\ &\leq C\sup_{t \in[0, T]}\eta^2e^{\alpha t}\|\sqrt{\rho}\dot{u}\|_{L^2}^2+\int_{0}^T\eta e^{\alpha t}\|\sqrt{\rho}\dot{u}\|_{L^2}^2dt\leq C.\la{lm36c}
	\end{aligned}
\end{equation}

Moreover, since  $\dot u|_{\partial\Omega}=0,$  \eqref{31}, \eqref{315},  and Poincar\'{e} inequality, we have   for any $p\in(2,\infty)$
\begin{equation*}\la{why8}
\|\na^2u\|_{L^p}^2+\|\na P\|_{L^p}^2\leq C\|\rho\dot{u}\|^2_{L^p}\leq C\bar{\rho}^2\|\dot u\|_{L^p}^2\leq C\|\na\dot{u}\|^2_{L^2},
\end{equation*}
which, coupled with  \eqref{32600}, gives 
\begin{equation*}\label{lm36d}
	\int_0^T\eta^2 e^{\alpha t}\left(\left\|\na^2 u\right\|_{L^p}^2+\left\|\na P\right\|_{L^p}^2\right)dt\leq C.
\end{equation*}
Combining this,  \eqref{lm36b} and \eqref{lm36c}, we deduce \eqref{lm36a}-\eqref{lm36aaa} directly. Hence, the proof of Lemma \ref{lm33} is finished. \hfill $\Box$

\subsection{The time-weighted gradient estimate of velocity in bounded domains}
\la{sec:bdry-linfty}

In this section, we establish the crucial $L^2(0,T;L^\infty)$-estimate for $t^{\frac12}\na u$ in  bounded domains. The key piont is  to obtain the following  point-wise estimate for $\na u$. Indeed,  our argument relies on the  localized Green
functions estimate
for  the linear Dirichlet Stokes problem (see \eqref{bdrylocal21}),  and local boundary regularity estimate for the homogeneous Stokes system (see \eqref{bdrylocal22}), both on  the  constructed  local \(C^2\) domains.

\begin{lemma}\la{lm:bdrylocal}
	Let $(\rho,u,P)$ be a smooth solution to \eqref{11}-\eqref{14} satisfying \eqref{15}. 
	Then there exists  positive constant $r_{\Om}$,
	such that 
    for any  \(x_0\in\Omega\), and \(r\in(0,r_{\Omega}]\),
     there is a bounded $C^2$
	domain $E_r(x_0)\subset \Om$ satisfying
	\begin{equation}\la{bdrylocal3a}
		x_0\in E_r(x_0)\subset \Om\cap B(x_0,10r),
	\end{equation}
	and   for any  $t\in[0,T]$,
\begin{equation}\la{bdrylocal3}
		|\na u(x_0,t)|
		\le
		Cr^{-1}\|\na u\|_{L^2(\Om)}
		+
		C\int_{E_r(x_0)}\frac{|\rho\dot u(y,t)|}{|x_0-y|}dy,
	\end{equation}
    where the constant $C>0$ depends  only on $\Om$.
\end{lemma}

\noindent{\it Proof.} For any  $t\in[0,T]$, $x_0\in\Om$,  we will divide the proof into two cases. 

\smallskip
\noindent
{\bf Case 1. Near-boundary points.}

Choose $x_b\in \partial\Om$ such that
\[
|x_0-x_b|=\operatorname{dist}(x_0,\partial\Om).\] For any  $r\in(0,r_0]$, $r_0$ is given  in  Lemma \ref{lm:geompatch}.

If   $\operatorname{dist}(x_0,\partial\Om)<r$, then by Lemma \ref{lm:geompatch}, there exists  a  bounded  $C^2$  domain  $D_{4r}(x_b)$  such that
\begin{equation}\la{bdrylocal9}
	x_0\in\Om\cap B(x_b,r)\subset D_{4r}(x_b)\subset \Om\cap B(x_b,6r)\subset \Om\cap B(x_0,10r),
\end{equation}
and
\begin{equation}\la{bdrylocal10}
	D_{4r}(x_b)\cap B(x_b,2r)=\Om\cap B(x_b,2r).
\end{equation}

We denote 
\[
E_r(x_0)\triangleq D_{4r}(x_b),
\]
and let $(v,\pi)\in H^1_0(E_r(x_0))^2\times L^2(E_r(x_0))$,
with $\int_{E_r(x_0)}\pi\,dx=0$,  be the unique solution of the following
local Dirichlet Stokes {  problem, given by
\cite[Theorem 4.6]{AmroucheGirault1994}.} 
\begin{equation}\la{bdrylocal21}
	\left\{
	\begin{aligned}
		-\Delta v+\na \pi&=-\rho\dot{u},  &&x\in E_r(x_0),\\
		\div v&=0,   &&x\in E_r(x_0),\\
		v&=0,   &&x\in \partial E_r(x_0).
	\end{aligned}
	\right.
\end{equation}

Now,  we  justify    \[	|\nabla v(x_0)|
\le C
\int_{E_r(x_0)}
\frac{|\rho\dot u(y)|}{|x_0-y|}\,dy,\]
where the constant $C$  depends only on $\Om$, independent of
$x_0$ and $r$.

In fact, let
\[ 
z=r^{-1}T_{x_b}(x),\quad
z_0=r^{-1}T_{x_b}(x_0),
\]
where  \(T_{x_b}(x)\) is given in   Lemma \ref{lm:geompatch}. Then \(D_{4r}(x_b)\) is transformed into
\[
\widetilde D_4(x_b,r)
=
r^{-1}T_{x_b}(D_{4r}(x_b)).
\]
By Lemma \ref{lm:geompatch}, {  these complete normalized} domains { have
common} {  graph} {  radius} {  and} {  Dini modulus,} bounded {  diameter} and measure,
and {  a uniform divergence right inverse. All these constants} depend only
on { $\Omega$} { and} { are} {  independent} {  of} {  $x_b,r$}.

Then, for \(z\in \widetilde D_4(x_b,r)\),  we set
\[
\widetilde v(z)
\triangleq
Q_{x_b}v\bigl(T_{x_b}^{-1}(rz)\bigr),
\quad
\widetilde \pi(z)
\triangleq
r\pi\bigl(T_{x_b}^{-1}(rz)\bigr),
\]
which implies 
\[
\left\{
\begin{aligned}
	-\Delta_z\widetilde v+\nabla_z\widetilde\pi
	&=
	\widetilde F,
	&& z\in \widetilde D_4(x_b,r),\\
	\div_z\widetilde v&=0,
	&& z\in \widetilde D_4(x_b,r),\\
	\widetilde v&=0,
	&& z\in \partial\widetilde D_4(x_b,r),
\end{aligned}
\right.
\]
where
\[
\widetilde F(z)
\triangleq
-r^2 Q_{x_b}
\bigl(\rho\dot u\bigr)
\bigl(T_{x_b}^{-1}(rz)\bigr).
\]

By Lemma \ref{thm:CK37} with
$L=L^*=D_\alpha(\delta_{\alpha\beta}I_2D_\beta)=\Delta_z$,
applied to $(\widetilde v,-\widetilde\pi)$ with $f=-\widetilde F$ and $g=0$,
there exists a unique Green function $(G,\Pi)$. By \eqref{CK32},
$G(\zeta,z)^\top=G(z,\zeta)$, and hence
\begin{equation*}\label{greenrepr2}
\begin{aligned}
\widetilde v(z)
&=-\int_{\widetilde D_4(x_b,r)}G(\zeta,z)^\top
\bigl(-\widetilde F(\zeta)\bigr)\,d\zeta\\
&=\int_{\widetilde D_4(x_b,r)}G(z,\zeta)\widetilde F(\zeta)\,d\zeta,
\quad \text{for a.e. }z\in\widetilde D_4(x_b,r).
\end{aligned}
\end{equation*}
and
\begin{equation}\label{greenrepr4}
	|\nabla_z G(z,\zeta)|+|\Pi(z,\zeta)|
	\le \frac{C}{|z-\zeta|},
	\quad \forall z,\zeta\in \widetilde D_4(x_b,r),\quad  z\neq \zeta.
\end{equation}
 Since $\widetilde F$ is bounded and $|z-\zeta|^{-1}$ is locally integrable in two dimensions, \eqref{greenrepr4} permits differentiation in distributions. By continuity, we obtain 
\begin{equation*}\label{greenrepr3}
	\nabla_z \widetilde v(z)=\int_{ \widetilde D_4(x_b,r)} \nabla_z G(z,\zeta)\widetilde F(\zeta)\,d\zeta,
	\quad \forall z\in  \widetilde D_4(x_b,r).
\end{equation*}
Combining this with  \eqref{greenrepr4}, yields
\begin{equation}\la{w1}
	|\nabla_z\widetilde v(z)|
	\le
	C
	\int_{\widetilde D_4(x_b,r)}
	\frac{|\widetilde F(\zeta)|}{|z-\zeta|}\,d\zeta,\quad \forall z\in  \widetilde D_4(x_b,r),
\end{equation}
here \(C\) is independent of \(x_b\) and  \(r\).

Since
\[
\nabla_z\widetilde v(z)
=
rQ_{x_b}(\nabla_x v(x))Q_{x_b}^{T},
\]
and \(Q_{x_b}\) is orthogonal, one can  check that 
\[
|\nabla_z\widetilde v(z)|
=
r|\nabla_x v(x)|,
\]
which, together with \eqref{bdrylocal9} and \eqref{w1}, indicates
\begin{equation}\la{w2}
	\begin{aligned}
		|\nabla_x v(x_0)|
		&\le
		Cr^{-1}
		\int_{\widetilde D_4(x_b,r)}
		\frac{
			r^2
			|\rho\dot u(T_{x_b}^{-1}(r\zeta))|
		}{
			|z_0-\zeta|
		}
		\,d\zeta .
	\end{aligned}
\end{equation}

Finally, changing variables
\[
y=T_{x_b}^{-1}(r\zeta),
\quad
dy=r^2\,d\zeta,
\]
and using the fact that \(Q_{x_b}\) is orthogonal, we get
\[
|z_0-\zeta|
=
r^{-1}
|T_{x_b}(x_0)-T_{x_b}(y)|
=
r^{-1}|x_0-y|.
\]
Then,  \eqref{w2} gives
\begin{equation}
	\begin{aligned}
		|\nabla v(x_0)|
		&\le
		Cr^{-1}
		\int_{D_{4r}(x_b)}
		\frac{
			|\rho\dot u(y)|
		}{
			r^{-1}|x_0-y|
		}
		\,dy  \\
		&=
		C
		\int_{D_{4r}(x_b)}
		\frac{|\rho\dot u(y)|}{|x_0-y|}\,dy\\ &=
		C
		\int_{E_r(x_0)}
		\frac{|\rho\dot u(y)|}{|x_0-y|}\,dy .
	\end{aligned}\la{why10}
\end{equation}
Here, the constant \(C\) is  independent of \(x_0\), \(x_b\),
and \(r\).

Denote
\[
w\triangleq u-v.
\]
By  \eqref{11} and \eqref{bdrylocal21},  we  derive  that  $(w,P-\pi)$ satisfies
\begin{equation}\la{bdrylocal22}
	\left\{
	\begin{aligned}
		-\Delta w+\na(P-\pi)&=0,\quad  x\in E_r(x_0),\\ \div w&=0, \quad  x\in E_r(x_0).
	\end{aligned}\right.
\end{equation}
Meanwhile, combining   \eqref{14}, \eqref{bdrylocal10} and $\eqref{bdrylocal21}_3$ yields 
\[
w=0\quad \text{on }\partial E_r(x_0)\cap B(x_b,2r).
\]

So, it suffices to prove  that 
$$|\na w(x_0)|\le Cr^{-1}\|\na u\|_{L^2(\Om)},$$
where the constant $C$  depends only on $\Om$. 

To this end, we take
\[\widetilde w(z)
\triangleq
Q_{x_b}w\bigl(T_{x_b}^{-1}(rz)\bigr), \quad q(x)
\triangleq(P-\pi)(x),
\quad
\widetilde q(z)
\triangleq
r\,q\bigl(T_{x_b}^{-1}(rz)\bigr), 
\] then $(\widetilde w,\widetilde q)$ solves the homogeneous Stokes system
\begin{equation}\la{hy0}
	\left\{	\begin{aligned}
	-\Delta_z \widetilde w+\na_z \widetilde q&=0,\quad  z\in\widetilde D_4(x_b,r),\\\div_z \widetilde w&=0,
	\quad z\in\widetilde D_4(x_b,r),\\\widetilde w&=0,
	\quad z\in\partial \widetilde D_4(x_b,r)\cap B(0,2).
	\end{aligned}\right.
\end{equation}

Set
\[
U_\beta\triangleq\widetilde D_4(x_b,r)\cap B(0,\beta),
\quad 1\le\beta\le2,
\]
and 
\[
\widetilde q_0=\widetilde q-(\widetilde q)_{\widetilde D_4(x_b,r)}, \quad
( \widetilde q)_{\widetilde D_4(x_b,r)}\triangleq\average{\widetilde D_4(x_b,r)} \widetilde q\,dy .
\]
In particular, since $x_0\in\Omega\cap B(x_b,r)$, one has
$z_0\in U_1$.

 By  Lemma  \ref{thm:Galdi-local-boundary-Stokes}  and Sobolev embedding, we obtain
\begin{equation}\label{eq:stokes-bootstrap}
	\begin{aligned}
		\|\widetilde{w}\|_{W^{2,4}(U_{5/4})}
		+\| \widetilde{q}_0\|_{W^{1,4}(U_{5/4})}
		&\leq C\|\widetilde{w}\|_{W^{1,4}(U_{3/2})}
		+C
		\|\widetilde q_0\|_{L^4(U_{3/2})}\\&\leq C\|\widetilde{w}\|_{W^{2,2}(U_{7/4})}
		+C\|\widetilde{q}_0\|_{W^{1,2}(U_{7/4})} 
			\\&\le C\|\widetilde w\|_{W^{1,2}(\widetilde D_4(x_b,r))}+C\|\widetilde q_0\|_{L^2(\widetilde D_4(x_b,r))}, 
	\end{aligned}
\end{equation}
Here, for $1\le a<b\le2$,
\[
\partial U_a\cap\partial\widetilde D_4(x_b,r)
\subset\partial\widetilde D_4(x_b,r)\cap\overline{B(0,a)}
\subset\subset\partial\widetilde D_4(x_b,r)\cap B(0,b),
\]
with compact inclusion relative to $\partial\widetilde D_4(x_b,r)$, and
\[
\operatorname{dist}\left(\overline{U_a},
\overline{\partial U_b\setminus\partial\widetilde D_4(x_b,r)}\right)\ge b-a.
\]
Thus Lemma \ref{thm:Galdi-local-boundary-Stokes} applies with
$\Omega=\widetilde D_4(x_b,r)$ and $m=0$, first with
$(\Omega_0,\Omega^0)=(U_{7/4},U_2)$ and exponent $2$, and then with
$(\Omega_0,\Omega^0)=(U_{5/4},U_{3/2})$ and exponent $4$, using the zero trace in \eqref{hy0}$_3$.
For the intermediate Sobolev embeddings, we multiply by a fixed smooth radial cutoff
supported in $B(0,7/4)$ and equal to one on $B(0,3/2)$, and extend by zero across
the artificial boundary inside $\widetilde D_4(x_b,r)$.
By \eqref{Dtilde-def} and the uniform $C^2$ bounds in Lemma \ref{lm:geompatch},
the complete domains $\widetilde D_4(x_b,r)$ have uniform boundary charts and
Sobolev extension bounds. Together with the fixed separation $1/4$,
this shows that the constants in \eqref{eq:stokes-bootstrap} are independent of $x_b$ and $r$.

In view of \eqref{hy0}$_3$     and  Lemma \ref{lm:local-poincare-normalized}, one can deduce  
\begin{equation*}\|\widetilde w\|_{W^{1,2}(\widetilde D_4(x_b,r))} \le C\|\na\widetilde w\|_{L^{2}(\widetilde D_4(x_b,r))},
\end{equation*}
Multiplying $\widetilde w$ by a fixed smooth radial cutoff supported in $B(0,5/4)$
and equal to one on $B(0,1)$, extending by zero inside $\widetilde D_4(x_b,r)$,
and using $W^{2,4}(\widetilde D_4(x_b,r))\hookrightarrow
C^1(\overline{\widetilde D_4(x_b,r)})$ and \eqref{eq:stokes-bootstrap}, we obtain 
 \begin{equation}\label{eq:unit-scale-gradient}
	|\nabla_z \widetilde{w}(z)|
	\le C	\|\na\widetilde w\|_{L^{2}(\widetilde D_4(x_b,r))}+C\|\widetilde q_0\|_{L^2(\widetilde D_4(x_b,r))},\quad \forall z\in U_1.
\end{equation}

Now, we control the pressure term on the right-hand side of \eqref{eq:unit-scale-gradient}. 

By Lemma \ref{div0},  there exists a 
 function $ h\in [H_0^1(\widetilde D_4(x_b,r))]^2$ satisfying \[ \operatorname{div} h=\widetilde q_0,\]   and
\begin{equation}\label{div3}
 	\|\na h\|_{L^2(\widetilde D_4(x_b,r))} \leq C\|\widetilde q_0\|_{L^2(\widetilde D_4(x_b,r))}.
 \end{equation}

Moreover, multiplying \eqref{hy0}$_1$ by $h$ and integrating the resulting equality, we have
\begin{equation*}
    \begin{aligned}\left|\int_{\widetilde D_4(x_b,r)} {\widetilde q_0}^2\,dy\right|&=\left|\int_{\widetilde D_4(x_b,r)} \widetilde q_0\,\diver h\,dy\right|=
\left|\int_{\widetilde D_4(x_b,r)}\nabla  \widetilde w:\nabla h\,dy\right|
\\&\le
\|\nabla  \widetilde w\|_{L^2(\widetilde D_4(x_b,r))}\|\nabla h\|_{L^2(\widetilde D_4(x_b,r))},
    \end{aligned}
\end{equation*}
which, along with \eqref{div3} implies
\begin{equation*}\label{eq:pressure-estimate}
	\| \widetilde q_0\|_{L^2(\widetilde D_4(x_b,r))}
	\le C\|\nabla  \widetilde w\|_{L^2(\widetilde D_4(x_b,r))}.
\end{equation*}
This and \eqref{eq:unit-scale-gradient} lead to  \begin{equation*}\label{hy2}
	|\nabla_z \widetilde{w}(z)| 
	\le C\|\nabla_z \widetilde{w}\|_{L^2(\widetilde D_4(x_b,r))},\quad \forall z\in U_1,
\end{equation*}
where  the constant $C$ depends only on $\Om$.

Scaling back the variables yields
\[
|\nabla_x w(x_0)|=\frac1r|\nabla_z\widetilde w(z_0)|,
\]
and 
\[
\|\nabla_z\widetilde w\|_{L^2(\widetilde D_4(x_b,r))}
=
\|\nabla_x w\|_{L^2(D_{4r}(x_b))}.
\]
Therefore,    
\begin{equation}\la{why3}
	|\na w(x_0)|\le Cr^{-1}\|\na w\|_{L^2(E_r(x_0))}.
\end{equation}

Testing $\eqref{bdrylocal22}_1$ by $v$, and using $\eqref{bdrylocal21}_2$, $\eqref{bdrylocal21}_3$,  one has
\begin{equation*}
	\int_{E_r(x_0)} \na w:\na v\,dx=0,	
\end{equation*}
which gives
\begin{equation}\la{why0}
	\int_{E_r(x_0)}|\na u|^2dx
	=
	\int_{E_r(x_0)}|\na w|^2dx+\int_{E_r(x_0)}|\na v|^2dx.
\end{equation}
Thus, \eqref{why3} and \eqref{why0} yield
\begin{equation*}
	|\na w(x_0)|\le Cr^{-1}\|\na u\|_{L^2(E_r(x_0))}\le Cr^{-1}\|\na u\|_{L^2(\Om)},
\end{equation*}
due to $E_r(x_0)\subset \Om. $ This and \eqref{why10} show that
\begin{equation}
	\la{hy101}	|\na u(x_0,t)|
	\le
	Cr^{-1}\|\na u\|_{L^2(\Om)}
	+
	C\int_{E_r(x_0)}\frac{|\rho\dot u(y,t)|}{|x_0-y|}dy.
\end{equation}
{\bf Case 2. Interior points.} 

 For any  $r\in(0, \frac{1}{2}\operatorname{diam}\Om)$, if $\operatorname{dist}(x_0,\partial\Om)\ge r$, we reset
\[
E_r(x_0)\triangleq B(x_0,r).
\]
Then $x_0\in E_r(x_0)\subset \Om\cap B(x_0,2r)\subset \Om\cap B(x_0,10r)$. For simplicity,  we still  denote $(v,\pi)$ as the solution to \eqref{bdrylocal21}.

 For $z\in B(0,1)$, we set
\[
\widetilde v(z)\triangleq v(x_0+rz),\quad
\widetilde \pi(z)\triangleq r\pi(x_0+rz),\quad
\widetilde F(z)\triangleq-r^2\rho\dot u(x_0+rz).
\]
 Then $(\widetilde v,\widetilde \pi)$ satisfies
\begin{equation}\la{hy4}
\left\{
\begin{aligned}
	-\Delta_z \widetilde v+\nabla_z \widetilde\pi
	&=\widetilde F,  &&z\in B(0,1),\\
	\div_z \widetilde v&=0, &&z\in  B(0,1),\\
	\widetilde v&=0, &&z\in\partial  B(0,1).
\end{aligned}
\right.	
\end{equation}
Similar as \eqref{w1},  we observe  
\[
|\nabla_z \widetilde v(0)|
\le
C\int_{B(0,1)}\frac{|\widetilde F(\zeta)|}{|\zeta|}\,d\zeta ,
\]
where the constant $C$  is independent of $x_0$ and $r$.

Note that 
\[
\nabla_z \widetilde v(0)=r\nabla_x v(x_0),
\]
then
\begin{equation}
\begin{aligned}\la{bdrylocal31}
	|\nabla_x v(x_0)|
	&\le
	Cr^{-1}\int_{B(0,1)}
	\frac{r^2|\rho\dot u(x_0+r\zeta)|}{|\zeta|}\,d\zeta  \\
	&=
	C\int_{B(x_0,r)}
	\frac{|\rho\dot u(y)|}{|x_0-y|}\,dy \\&=
	C\int_{E_r(x_0)}
	\frac{|\rho\dot u(y)|}{|x_0-y|}\,dy .
\end{aligned}
\end{equation}
Here the constant $C$ is independent of $x_0$ and $r$.

Denote
\[
w\triangleq u-v,
\]
and for $z\in B(0,1)$, \[ \widetilde{w}(z)\triangleq\frac1r w(x_0+rz),\quad
\widetilde{q}(z)\triangleq (P-\pi)(x_0+rz).
\]
 Then  $(\widetilde{w}, \widetilde{q})$ satisfies  by  \eqref{11} and \eqref{hy4},
\begin{equation*}
\la{hy5}
	\left\{	\begin{aligned}
		-\Delta_z \widetilde{w}+\nabla_z \widetilde{q}&=0, &&  z\in  B(0,1),\\
		\diver_z \widetilde{w}&=0, && z\in  B(0,1).
	\end{aligned}\right.
\end{equation*}

By virtue of Lemma \ref{lm:stokes-int-w2p-lp} and Poincar\'e  inequality on the ball, we find \[
\|\widetilde w-\widetilde w_{B(0,1)}\|_{W^{2,4}(B(0,1/2))}
\le C\|\widetilde w-\widetilde w_{B(0,1)}\|_{L^4(B(0,1))}\le
C\|\na_z\widetilde  w\|_{L^2(B(0,1))},
\]
which, together with  the Sobolev embedding   $W^{2,4}(B(0,1/2))\hookrightarrow C^{1}(\overline{B(0,1/2)})$, 
 implies 
\begin{equation*}\la{locC110}
	|\na_z \widetilde w(z)|\le C\|\na_z \widetilde w\|_{L^2(B(0,1))},\quad\forall z\in B(0,1/2).
\end{equation*}

Rescaling this back to the  variables, we have
\[
|\nabla_x w(x)|
=
|\nabla_z \widetilde{w}(z)|,\]
and
\[
\|\nabla_z \widetilde{w}\|_{L^2(B(0,{1}))}^2=
\int_{B(0,{1})}|\nabla_x w(x_0+rz)|^2\,dz
=
r^{-2}\int_{B(x_0,r)}|\nabla_x w(x)|^2\,dx.\]
Thus, 
\begin{equation*}\la{locC11}
	|\na w(x_0)|\le Cr^{-1}\|\na w\|_{L^2(E_r(x_0))},
\end{equation*}
 which, combined with \eqref{why0}, yields
\begin{equation}\la{W}
|\na w(x_0)|\le Cr^{-1}	\|\na w\|_{L^2(E_r(x_0))}\le Cr^{-1}\|\na u\|_{L^2(\Om)}.
\end{equation}

Finally, it follows  from  \eqref{bdrylocal31} and  \eqref{W}  that 
\begin{equation*}\la{why101}
	|\na u(x_0,t)|
	\le
	Cr^{-1}\|\na u\|_{L^2(\Om)}
	+
	C\int_{E_r(x_0)}\frac{|\rho\dot u(y,t)|}{|x_0-y|}dy.
\end{equation*}

Taking $r_\Om\triangleq\min\{r_0, \frac12\operatorname{diam}\Om\}$ and combining the two cases above,   we complete  the  proof  of Lemma \ref{lm:bdrylocal}. \hfill $\Box$

Armed with  Lemma \ref{lm:bdrylocal}, we proceed to derive the  crucial estimate \(t^{1/2}\nabla u\in L^2(0,T;L^\infty(\Omega))\) in  bounded domains.
\begin{lemma}\la{lm:bdrylemma23}
	Let $(\rho,u,P)$ be a smooth solution to \eqref{11}-\eqref{14} on $\Om\times[0,T]$ satisfying \eqref{15}.
	Then there exists a positive constant $C$ depending only on $\bar\rho$, $\left\|u_0\right\|_{L_\sigma^2}$,    and $\Omega$  such that
	\begin{equation}\la{bdrylemma23b}
		\int_0^Tt\|\na u\|_{L^\infty}^2dt\le C.
	\end{equation}
\end{lemma}

\noindent{\it Proof.} For any  $t\in[0,T]$, $x_0\in\Om$, and $0<r\le r_{\Om},$ 
we set
  $$E\triangleq \Omega\cap B(x_0,10r), \quad a_r\triangleq\frac{1}{|B(x_0,10r)|}\int_{B(x_0,10r)}\widetilde{\dot{u}}(y,t)dy,$$
   where $r_{\Om}$ is given in Lemma \ref{lm:bdrylocal} and $\widetilde{\dot{u}}$ denotes the zero extension of $\dot{u}$ outside \(\Omega\). 
  
 First, \eqref{bdrylocal3} and $E_r(x_0)\subset E$ give 
\begin{equation}\la{bdrylemma23c}
	\begin{aligned}
		|\na u(x_0,t)|
		\le
		Cr^{-1}\|\na u\|_{L^2(\Om)}
		+CK_1(x_0,t,r)+CK_2(x_0,t,r),
	\end{aligned}
\end{equation}
where
\[
K_1(x_0,t,r)\triangleq \int_E\frac{\rho|\dot u-a_r|}{|x_0-y|}dy,
\quad
K_2(x_0,t,r)\triangleq \int_E\frac{\rho|a_r|}{|x_0-y|}dy.
\]

By the Poincar\'e inequality on the ball 
\[
\|\widetilde{\dot{u}}- a_r\|_{L^2(B(x_0,10r))}\le Cr\|\na \widetilde{\dot{u}}\|_{L^2(B(x_0,10r))}= Cr\|\na {\dot{u}}\|_{L^2(E)},
\]
and the Gagliardo-Nirenberg inequality, 
we have   for any $p>2$,
\begin{equation*}
	\|\widetilde{\dot{u}}-a_r\|_{L^p(B(x_0,10r))}
	\le C\|\widetilde{\dot{u}}-a_r\|_{L^2(B(x_0,10r))}^{2/p}\|\na\widetilde{\dot{u}}\|_{L^2(B(x_0,10r))}^{1-2/p}\le
	Cr^{2/p}\|\na \dot u(\cdot,t)\|_{L^2(E)}.
\end{equation*}
This,  together with 
\begin{equation*}\la{bdrylemma23e}
	\bigl\||x_0-\cdot|^{-1}\bigr\|_{L^{p'}(B(x_0,10r))}\le Cr^{2/p'-1}, \quad1/p+1/{p'}=1
\end{equation*}
and \eqref{31}, gives 
\begin{equation}\la{bdrylemma23f}
	K_1(x_0,t,r)
	\le
	\bar\rho \|\widetilde{\dot{u}}-a_r\|_{L^p(B(x_0,10r))}
	\bigl\||x_0-\cdot|^{-1}\bigr\|_{L^{p'}(B(x_0,10r))}
	\le
	C\bar\rho\, r \|\na \dot u(\cdot,t)\|_{L^2}.
\end{equation}

We next estimate $K_2$. 

By the interpolation inequality
in the Lorentz space   and \eqref{31}, we observe 
\begin{equation}\label{why5}
\begin{aligned}
	K_2(x_0,t,r)
&=
|a_r|\int_E \frac{\rho(y,t)}{|x_0-y|}\,dy\\&\le 
|a_r|\|\rho\|_{L^{2,1}(E)}\bigl\||x_0-\cdot|^{-1}\bigr\|_{L^{2,\infty}(E)}\\&\le
C\|\rho\|_{L^{\infty}(E)}^{1/2}\|\rho\|_{L^{1}(E)}^{1/2}|a_r|\\&\le
C\bar\rho^{1/2}\|\rho\|_{L^1(E)}^{1/2}|a_r|.
\end{aligned}
\end{equation}
Note that 
\begin{equation}\la{why4}
	\|\rho\|_{L^1(E)}|  a_r|^2
	=
	\int_E \rho |  a_r|^2dy
	\le
	2\int_E \rho|\dot u-a_r|^2dy
	+
	2\int_E \rho|\dot u|^2dy.
\end{equation}
And  the
Poincar\'e inequality on the ball gives 
\[
\|\dot u-a_r\|_{L^2(E)}\le \|\widetilde{\dot{u}}- a_r\|_{L^2(B(x_0,10r))}\le Cr\|\na \widetilde{\dot{u}}\|_{L^2(B(x_0,10r))}= Cr\|\na {\dot{u}}\|_{L^2(E)},
\]
which, together with \eqref{31}, implies  
\begin{equation}\la{i}
	\int_E \rho|\dot u-a_r|^2dy\le C\bar\rho\, r^2\|\na \dot u\|_{L^2}^2.
\end{equation}
Then, \eqref{why4} and \eqref{i} indicate
\[
\|\rho\|_{L^1(E)}^{1/2}|a_r|
\le
C\bar\rho^{1/2}r\|\na \dot u\|_{L^2}+C\|\sqrt{\rho}\dot u\|_{L^2}.
\]
The combination of  this and \eqref{why5} leads to 
\begin{equation}\la{bdrylemma23g}
	K_2(x_0,t,r)\le C\bar\rho\, r\|\na \dot u\|_{L^2}+C\bar\rho^{1/2}\|\sqrt{\rho}\dot u\|_{L^2}.
\end{equation}
Inserting  \eqref{bdrylemma23f} and \eqref{bdrylemma23g} into \eqref{bdrylemma23c}, we find
\begin{equation}\la{bdrylemma23h}
	\|\na u\|_{L^\infty}
	\le
	Cr^{-1}\|\na u\|_{L^2}+C\bar\rho\,r\|\na \dot u\|_{L^2}+C\bar\rho^{1/2}\|\sqrt{\rho}\dot u\|_{L^2}.
\end{equation}

{ If} { $\bar\rho=0$}, {  \eqref{33}} { and} {  smoothness} {  give} {  $\nabla u=0$,
so} { \eqref{bdrylemma23b}} {  holds. Suppose $\bar\rho>0$.
If $\|\nabla u\|_{L^2}=0$}, then {  $u=0$} {  and
$\|\nabla u\|_{L^\infty}=0$ at this time.
When both $\|\nabla u\|_{L^2}$ and $\|\nabla\dot u\|_{L^2}$ are positive, set}
\[
R\triangleq \left(\frac{\|\na u\|_{L^2}}{\bar\rho \|\na \dot u\|_{L^2}}\right)^{1/2}.
\]

If   $\|\na u\|_{L^2}=0$, since $u|_{\partial\Omega}=0,$ we find  $u=0$ a.e. in $\Om$, then  \eqref{bdrylemma23b} holds directly.
 
If  $\|\na u\|_{L^2}>0$,  $\|\na \dot u\|_{L^2}>0$, and  $R\leq r_{\Om}$,  choosing $r=R$ in \eqref{bdrylemma23h} yields
\begin{equation}\la{bdrylemma23i0}
\|\na u \|_{L^\infty}
\le
C\bar\rho^{1/2}\|\na u\|_{L^2}^{1/2}\|\na \dot u\|_{L^2}^{1/2}+C\bar\rho^{1/2}\|\sqrt{\rho}\dot u\|_{L^2}.
\end{equation}

If  $\|\na u\|_{L^2}>0$,     $\|\na \dot u\|_{L^2}>0$, $R>r_{\Om}$, or $\|\na u\|_{L^2}>0$,  $\|\na \dot u\|_{L^2}=0$,  then 
\begin{equation*}\la{why6}
	\|\na u\|_{L^2}>\bar\rho r_\Om^2 \|\na \dot u\|_{L^2},
\end{equation*}
which combined with \eqref{bdrylemma23h} ($r=r_{\Om}$)     leads to 
\begin{equation}\label{hy60}
	\|\na u \|_{L^\infty}
	\le
	C{r_\Omega}^{  {-1}}\bar\rho^{1/2}\|\na u\|_{L^2}+C\bar\rho^{1/2}\|\sqrt{\rho}\dot u\|_{L^2}.
\end{equation}

Multiplying \eqref{11}$_2$ by $u$ and integrating the resulting equality over $\Om$, give 
\begin{equation*}
	\int_\Om \left|\na u\right|^2dx=-\int_\Om \rho \dot u\cdot udx\leq\|\rho\dot u\|_{L^2}\|u\|_{L^2}\leq C\|\rho\dot u\|_{L^2}\|\na u\|_{L^2}\leq\frac12\|\na u\|_{L^2}^2+C\bar\rho\|\sqrt{\rho}\dot{u}\|_{L^2}^2,
\end{equation*}
where we have used \eqref{31} and the
Poincar\'e inequality. This implies 
\begin{equation}\label{hy6}
	\|\na u\|_{L^2}\leq C \bar\rho^{1/2}\|\sqrt{\rho}\dot{u}\|_{L^2}.
\end{equation}
Then, we have by \eqref{31},  \eqref{hy60} and  \eqref{hy6}
\begin{equation*}
	\|\na u \|_{L^\infty} 
	\le
	C\|\sqrt{\rho}\dot{u}\|_{L^2},
\end{equation*}
which together with  \eqref{bdrylemma23i0}, yields 
\begin{equation}\la{hy7}
	\|\na u \|_{L^\infty}
	\le
	C\left(\|\na u\|_{L^2}^{\frac12}\|\na \dot u\|_{L^2}^{\frac12}+\|\sqrt{\rho}\dot u\|_{L^2}\right).
\end{equation}

 Finally,   we get by  \eqref{33}, \eqref{313},  \eqref{32600}, and 
 \eqref{hy7}
\begin{equation*}
	\begin{aligned}
		\int_0^T t&\|\nabla u \|_{L^\infty}^2 d t \\\leq & C \int_0^T t \|\na u\|_{L^2}\|\na \dot u\|_{L^2} d t+C \int_0^T t \|\sqrt{\rho}\dot u\|_{L^2}^2 d t \\
		\leq & C\left(\int_0^T \|\na u\|_{L^2}^2 d t\right)^{1 / 2}\left(\int_0^T t^2 \|\na \dot u\|_{L^2}^2 d t\right)^{1 / 2}  +C \int_0^T t \|\sqrt{\rho}\dot u\|_{L^2}^2d t\\\leq&C+ C\int_0^{\eta(T)} \left(\eta^2 (t)\|\na \dot u\|_{L^2}^2+\eta (t)\|\sqrt{\rho}\dot u\|_{L^2}^2\right) d t+ C\int_{\eta(T)}^T \left( e^{\alpha t}\|\na \dot u\|_{L^2}^2+e^{\alpha t}\|\sqrt{\rho}\dot u\|_{L^2}^2\right) d t\\\leq&C,
	\end{aligned}
\end{equation*}
which shows  \eqref{bdrylemma23b} directly. So, the proof of Lemma \ref{lm:bdrylemma23} is completed. \hfill $\Box$

\section{Proof of our main results}\la{4}
With  all the a priori estimates in Section \ref{3}, we are now ready to prove
Theorems \ref{tm1} and \ref{tm2}.

\subsection{Proof of Theorem \ref{tm1}}
In this subsection, we shall  construct suitable approximate solutions and then prove the global  existence and regularity of  weak solutions to \eqref{11}–\eqref{14} by passing to the limit. 

We will divide  this proof   into two steps.

{\it Step 1. Construction of approximate solutions.} For the initial data $(\rho_0, u_0)$ defined in Theorem \ref{tm1}, we can construct the following smooth initial data $(\rho_0^N(x),u_0^N(x))$, for all sufficiently large $N\in\mathbb{Z}_+$, satisfying 
\begin{equation*}\label{470}
	\begin{cases}
		(\rho_0^N(x),u_0^N(x))\in C^{\infty}(\bar{\Om})\times C^\infty_{0,\sigma}(\Om),\\	\frac{1}{N}\leq\rho^N_0\leq 2\bar{\rho},\,\, \operatorname{div} u^N_0=0,\\
		\left\|\rho_{0}^N\right\|_{L^\infty(\Om)}\leq2\bar{\rho}, \,\,\left\|u_0^N\right\|_{L^2(\Om)}\leq C \left\|u_0\right\|_{{L^2(\Om)}},\\
		\lim\limits_{\substack{N \rightarrow \infty}}\left(\left\|\rho_0^{N}-\rho_0\right\|_{L^1(\Om)}+\left\|u_0^{N}-u_0\right\|_{L^2(\Om)}\right)=0.
	\end{cases}
\end{equation*}
Here $\rho_0^N$ is obtained by nonnegative mollification of the zero extension of $\rho_0$,
followed by addition of $N^{-1}$, so that
$\rho_0^N\leq\bar{\rho}+N^{-1}\leq2\bar{\rho}$ for all sufficiently large $N$.
Then, applying the classical strong solution theory (see \cite{Ladyzhenskaya1975}),
 there exists a unique global smooth solution $(\rho^N,u^N,P^N)$ to \eqref{11}–\eqref{14}, with the initial data $(\rho_0^N,u_0^N)$. Moreover, all a priori estimates in Section \ref{3} with $(\rho, u, P)$ replaced by $(\rho^{N}, u^{N}, P^{N})$ and $\bar{\rho}$ by $2\bar{\rho}$ hold with constants and $\alpha$ independent of $N$. 

{\it Step 2. Passing to the  limit.} By taking an appropriate subsequence,   still denoted by $(\rho^N, u^N, P^N)$ and applying standard  arguments (see \cite[Section 2.3]{Lions_1996}), we   conclude that   problem \eqref{11}-\eqref{14} admits a global weak  solution  $(\rho,  u, P)$  defined on $\Om \times [0, T]$, satisfying  as $N\rightarrow\infty$,
\begin{equation}\la{412}
	\rho^{N} \rightharpoonup \rho \quad\text {weakly-$*$} \text { in } L^{\infty}\left(\Om\times(0, T)\right),
\end{equation}
\begin{equation*}\la{413}
	\rho^{N} \rightarrow \rho \quad\text {in } C\left([0, T) ; L^p\left(\Omega\right)\right), \text { for any } 1\leq p <\infty,
\end{equation*}
\begin{equation*}\la{44311110}
	\rho^N u^{N} \rightarrow {\rho} u \quad\text{in } C\left([0, T] ; L^2_w\left(\Omega\right)\right),
\end{equation*}
\begin{equation*}\label{compact-kinetic}
\sqrt{\rho^N}u^N\to\sqrt{\rho}u\quad
 \text{in }L^2\bigl(\Om\times(0, T)\bigr),
\end{equation*}
\begin{equation}\la{nap}
\left(\na^2 u^N, \na P^N\right)\rightharpoonup\left( \na^2u, \na P\right)\,
\text{weakly in }L^2(\tau, T;L^p(\Om)), \text { for any } 0<\tau<T, 2\leq p <\infty,
\end{equation}
and
\begin{equation*}\label{compact-velocity}
u^N\rightharpoonup u
\quad\text{weakly in }L^2(0, T; H^1_{0,\sigma}).
\end{equation*}

By \eqref{31}, $0\leq\rho^N\leq\bar{\rho}+N^{-1}$;
hence \eqref{412} yields $0\leq\rho\leq\bar{\rho}$.

In particular,
\[
\rho^N u^N\otimes u^N
=(\sqrt{\rho^N}u^N)\otimes(\sqrt{\rho^N}u^N)
\longrightarrow \rho u\otimes u
\quad\text{in }L^1(\Om\times(0, T)),
\]
as $N\rightarrow\infty$. Moreover, it follows from   \eqref{313},  \eqref{lm36a}, \eqref{lm36aa}, and Sobolev embedding that
 for any  $t\in[\tau, T)$,
\begin{equation}\begin{aligned}\la{ii}
		\|&\na\left(u^N\cdot\nabla u^N\right)\|_{L^2(\tau, T;L^2)}\\&\le
		C\|e^{\frac{\alpha}{2}t}\na u^Ne^{-\frac{\alpha}{2}t}\|_{L^4(\tau, T;L^4)}^2+C
		\|\nabla^2 u^N\|_{L^2(\tau, T;L^2)}\|u^N\|_{L^\infty(\tau, T;L^\infty)}\\&\le C	\|e^{\frac{\alpha}{2}t}\na u^N\|_{L^\infty(\tau, T;L^4)}^2+C
		\|\nabla^2 u^N\|_{L^2(\tau, T;L^2)}\|u^N\|_{L^\infty(\tau, T;W^{1,4})}\\&\le C.
	\end{aligned}
\end{equation}

Combining \eqref{ii},  \eqref{32600},  Poincar\'e inequality, and 
\begin{equation*}
	\begin{aligned}
\|\na u^N_t\|_{L^2(\tau, T;L^2)}\le \|\na\dot u^N\|_{L^2(\tau, T;L^2)}+\|\na(u^N\cdot\nabla u^N)\|_{L^2(\tau,T;L^2)}
	\end{aligned}
\end{equation*}
 gives \begin{equation}
 	\la{iii}\partial_tu^N \in L^2(\tau, T; H^1).
 \end{equation}   
Then, \eqref{iii},   \eqref{lm36a}, 
 \eqref{lm36aaa},
 and  the Aubin-Lions-Simon
theorem  yield \[u^N\to u\quad\text{in }L^2(\tau,T;H^1(\Omega)).\]
So,  we have by  \eqref{iii} and \eqref{lm36aaa}
\[\dot u^N\rightharpoonup \dot u
\quad\text{weakly in }L^2(\tau,  T;H^1_0(\Om)).\]

Applying  this,  \eqref{0} and  \eqref{nap}  to
 \eqref{00} yields that the momentum equation holds in the strong
sense, i.e., \[\rho\dot{u}-\Delta u+\na P=0\quad\text{in }  L^2(\tau,T;L^p(\Om)), \quad \text {for any } 2\leq p <\infty.\]

Furthermore, for every $0<\tau<T$ and $2< p<\infty$, Lemmas \ref{lemma 3.40}--\ref{lm33} give
\begin{equation}
 \eta^{\frac{1}{2}}e^{\frac{\alpha}{2}t} \na u^N \rightarrow 	 \eta^{\frac{1}{2}}e^{\frac{\alpha}{2}t} \na u \quad\text  {weakly-$*$ in } L^\infty\left(0,T; L^2\left(\Omega\right)\right),
\end{equation}
\begin{equation}
	\eta^{\frac{1}{2}} e^{\frac{\alpha}{2}t} \sqrt{\rho^N}\dot{u}^N \rightarrow  	\eta^{\frac{1}{2}} e^{\frac{\alpha}{2}t} \sqrt{\rho}\dot{u}\quad \text  {weakly  in } L^2\left(\Omega\times(0,T)\right),
\end{equation}
\begin{equation}
 \eta e^{\frac{\alpha}{2}t} \sqrt{\rho^N}\dot{u}^N \rightarrow \eta e^{\frac{\alpha}{2}t}\sqrt{\rho}\dot{u} \quad\text  {weakly-$*$ in } L^\infty\left(0,T; L^2\left(\Omega\right)\right),
\end{equation}
\begin{equation}
\eta e^{\frac{\alpha}{2}t}\na \dot u^N\rightarrow \eta e^{\frac{\alpha}{2}t}\na \dot u\quad\text  {weakly  in } L^2\left(\Omega\times(0,T)\right),
\end{equation}\begin{equation}
t^{\frac12}\na   u^N\rightarrow t^{\frac12}\na   u\quad\text  {weakly-$*$  in } L^2\left(0,T;L^\infty(\Om)\right),
\end{equation}
\begin{equation}\la{416}
 \eta e^{\frac{\alpha}{2}t}	\left(\nabla u^{N}, P^{N}\right) \rightharpoonup \eta e^{\frac{\alpha}{2}t}\left(\nabla u, P\right)
	\quad\text{weakly in }L^2\left(\tau,T;W^{1,p}(\Om)\right),
\end{equation}
and
\begin{equation}\la{417}
 \eta e^{\frac{\alpha}{2}t}	\left(\nabla u^{N}, P^{N}\right) \rightharpoonup \eta e^{\frac{\alpha}{2}t}\left(\nabla u, P\right)
	\quad\text{ weakly-$*$ in } L^{\infty}\left(\tau,T;W^{1,2}(\Om)\right).
\end{equation}

Finally,   by  \eqref{417}, and taking the  limit of  $\int_\Omega P^N(t,x)\,dx=0$ as $N\rightarrow \infty$ gives
\[
\int_\Omega P(t,x)\,dx=0
\quad\text{for a.e. }t\in(\tau,\infty).
\]
Since  $\tau$ is arbitrary, we conclude that 
\[
\int_\Omega P(t,x)\,dx=0
\quad\text{for a.e. }t>0.
\]

By a lower-semicontinuity argument and  \eqref{412}-\eqref{417}, \eqref{31}, \eqref{33},  \eqref{313},  \eqref{32600}, \eqref{lm36a},  \eqref{lm36aa}, \eqref{bdrylemma23b},  
 we obtain   that $(\rho,  u, P)$ is a global weak solution, satisfying \eqref{150}–\eqref{1500}  in Theorem \ref{tm1}. Thus, we  complete the proof of Theorem \ref{tm1}. \hfill $\Box$

\subsection{Proof of Theorem \ref{tm2}}\label{4.2} 
For the case  $0<\rho_0\in L^\infty(\Om)$  bounded away from zero and $u_0 \in L^2_\sigma$, we will  establish the  uniqueness  of weak solutions in bounded domains as in \cite[Section 5.2]{wei_2025} located in the whole space.  To this end, we consider a sequence $\{u_{0}^N \}\subset H^1\left(\Om\right)$ such that $u_{0}^N \rightarrow u_0$ in $L^2\left(\Om\right)$ as $N \rightarrow \infty$. Since $u_{0}^N$ is more regular, the global well-posedness of strong solutions $(\rho^{(N)}, u^{(N)})$  to \eqref{11} with initial data  $(\rho_0, u_{0}^N)$  has already been established in \cite{DanchinMucha2019}. Then, we prove that $\rho^{(N)} \rightarrow \rho$ and $u^{(N)} \rightarrow u$ as $N \rightarrow \infty$ in some sense. Hence, the uniqueness of weak solution follows from the uniqueness of limits.

 We first give the following key stability  lemma.

\begin{lemma}[Stability  lemma]\label{prop:stability-bdd}
For 	 constants $\underline{\rho}$, $\bar{\rho}$ 	and   $T>0$, assume that the initial data $(\rho_0, u_0, \hat{u}_0)$ satisfy
	\begin{equation}\la{underrho}
		0<\underline{\rho}\leq\rho_0\leq\bar\rho, \quad u_0\in L^2_\sigma, \quad \hat u_0\in H^1_{0,\sigma}.
	\end{equation} Let  $(\rho, u,  P)$ and $(\hat\rho,\hat u, \hat P)$ (with $\int_\Om \left(P, \hat{P}\right) dx=0$) be  weak solutions of \eqref{11}-\eqref{14} on $[0,T]\times\Omega$
	with initial data $(\rho_0,u_0)$ and  $(\rho_0,\hat u_0)$, respectively.
	 Assume in addition that
	\begin{itemize}
		\item  $(\rho,u)$ and $(\hat\rho,\hat u)$ satisfy the energy inequality \eqref{y};
		\item $\sqrt{\hat\rho}\,\hat u_t\in L^2(0,T;L^2)$;
		\item there exists a constant $C_0>0$ such that
		\begin{equation}\la{eq:strong-regularity-bdd}
			\|\nabla \hat u\|_{L^2(0,T;L^2)}
		+\|t^{1/2}\nabla \hat u\|_{L^2(0,T;L^\infty)}
		+\|t\dot{\hat u}\|_{L^\infty(0,T;L^2)}
		+\|t\nabla \dot{\hat u}\|_{L^2(0,T;L^2)}
		\le C_0,
		\end{equation}
     where  $\dot{\hat u}\triangleq {\hat u}_t+{\hat u}\cdot\na {\hat u}.$
	\end{itemize}
	Then there exists some constant $C_1$ depending only on $\underline\rho, \bar{\rho}, \Om, C_0$
	 such that
	\begin{equation}\label{eq:stability-bdd}
		\begin{aligned}
			\sup_{0\le t\le T}&\left\|\sqrt{\rho}(u-\hat u)(t)\right\|_{L^2}
			+\left\|\na (u-\hat u)\right\|_{L^2(0,T;L^2)}\\&+
			\left\|t^{-1}(\rho-\hat\rho)\right\|_{L^4(0,T;W^{-1,4})}
			\le
			C_1 \left\|\sqrt{\rho_0}(u_0-\hat u_0)\right\|_{L^2}.
		\end{aligned}
	\end{equation}
\end{lemma}
{\it Proof.}
Denote
	\[
	\rho^\delta\triangleq\rho-\hat\rho,\qquad
	u^\delta\triangleq u-\hat u,\qquad
	P^\delta\triangleq P-\hat P.
	\]
Then it is clear that  $(\rho^\delta,u^\delta)$ satisfies
	\begin{equation}\label{eq:delta-system-bdd}
		\left\{
		\begin{aligned}
			&\partial_t \rho^\delta+\div(\rho^\delta\,\hat u)+\div(\rho\,u^\delta)=0,\\
			&\partial_t(\rho u^\delta)+\div(\rho u\otimes u^\delta)-\Delta u^\delta+\nabla P^\delta
			=
			-\rho^\delta\,\dot{\hat{u}}-\rho\,u^\delta\cdot \nabla \hat u,\\
			&\div u^\delta=0,\\
			&(\rho^\delta,u^\delta)|_{t=0}=(0,u_0-\hat u_0),
		\end{aligned}
		\right.
	\end{equation}
where  $\dot{\hat u}\triangleq {\hat u}_t+{\hat u}\cdot\na {\hat u}.$
By $\eqref{11}_1$ in the renormalized sense and \eqref{underrho}, we have 
\begin{equation}\la{udrho}
\underline\rho\leq\rho(x,t),\, \hat\rho(x,t)\leq\bar\rho
\quad\text{for a.e. }(x,t)\in \Om \times [0,T],
\end{equation}
which together with $\sqrt{\hat\rho}\,\hat u_t\in L^2(0,T;L^2)$ implies $\hat u_t\in L^2(0,T;L^2)$. This and $\sqrt{\hat\rho}\hat u\in L^\infty(0,T;L^2)$ yield  $\hat u\in C([0,T];L^2)$.

Testing  \eqref{11}$_2$  against \(\hat u\) for \((\rho,u)\),  followed by the energy inequalities for \((\rho,u)\) and \((\hat\rho,\hat u)\),  and  the approximation argument as  Lemmas 3.1–3.4 and Lemma 4.1 in \cite{crin2025},  gives
	\begin{equation}\label{eq:energy-delta-1}
		\begin{aligned}
				\frac12&\|\sqrt{\rho}\,u^\delta(\cdot,t)\|_{L^2}^2
			+
			\int_0^t \|\nabla u^\delta(\cdot,s)\|_{L^2}^2\,ds
			\\&\le
			\frac12\|\sqrt{\rho_0}(u_0-\hat u_0)\|_{L^2}^2
			-\int_0^t\int_\Omega (\rho\,u^\delta\cdot \nabla \hat u)\cdot u^\delta\,dx \,ds-\int_0^t\langle \rho^\delta(\cdot,s),\,\dot{\hat{u}}(\cdot,s)\cdot u^\delta(\cdot,s)\rangle\,ds\\&\triangleq\frac12\|\sqrt{\rho_0}(u_0-\hat u_0)\|_{L^2}^2
		    + \int_0^t L_1(s)\,ds+\int_0^tL_2(s)\,ds.
		\end{aligned}
\end{equation}	

We will  estimate  terms on the right-hand side  \eqref{eq:energy-delta-1} as follows.

First,  it follows from  \eqref{udrho}  the  Gagliardo-Nirenberg inequality that
\begin{equation}
	\begin{aligned}
	\biggl|\int_0^tL_1(s)\,ds\biggr|
	&\le \bar{\rho}\int_0^t \|u^\delta(\cdot,s)\|_{L^4}^2 \|\nabla \hat u(\cdot,s)\|_{L^2}\,ds\\
	&\le
	C\bar{\rho} \int_0^t\|u^\delta(\cdot,s)\|_{L^2}\|\nabla u^\delta(\cdot,s)\|_{L^2}\|\nabla \hat u(\cdot,s)\|_{L^2}\,ds\\
	&\le
	C\bar{\rho} \underline{\rho}^{-{1}/{2}}\int_0^t\|\sqrt{\rho}\,u^\delta(\cdot,s)\|_{L^2}\|\nabla u^\delta(\cdot,s)\|_{L^2}\|\nabla \hat u(\cdot,s)\|_{L^2}\,ds\\
	&\le
	\frac12 \int_0^t\|\nabla u^\delta(\cdot,s)\|_{L^2}^2\,ds
	+
	C \int_0^t \|\sqrt{\rho}\,u^\delta(\cdot,s)\|_{L^2}^2 \|\nabla \hat u(\cdot,s)\|_{L^2}^2\,ds.
	\label{eq:II-est}
\end{aligned}
\end{equation}

Then,	inserting \eqref{eq:II-est} into \eqref{eq:energy-delta-1},  we have  by  Gronwall's inequality and \eqref{eq:strong-regularity-bdd}
	\begin{equation}\label{eq:D-basic}
		D^2(t)
		\le
		C \int_0^t |L_2(s)|\,ds + C D^2(0),
		\qquad \forall \, t\in [0,T],
	\end{equation}
	where $C>0$ depends only on $\underline\rho,\bar\rho, \Omega, C_0,$ and we set, for any $t\in[0,T]$,  
	\begin{equation}\label{eq:D-def}
		D(t)\triangleq\sup_{0\le s\le t}\|\sqrt{\rho}\,u^\delta(\cdot,s)\|_{L^2}
		+
		\|\nabla u^\delta(\cdot,s)\|_{L^2(0,t;L^2)}.
\end{equation}

To estimate $\int_0^tL_2(s)\,ds$,  we mainly employ similar idea as \cite[Section 5.2]{wei_2025}. The key point is to control the bound for   $\|t^{-1}\rho^\delta\|_{L^4(0,t; W^{-1,4})}$. Indeed, for any fixed  $t\in(0,T]$ and $\varphi\in C_0^1(\Omega)$, Lemma \ref{lem:transport-bdd} shows that there exists   a unique  solution satisfying $\phi(s,x)$ 	\begin{equation}\label{eq:back-transport0}
	\left\{
	\begin{aligned}
		&\partial_s\phi+\hat u\cdot\nabla\phi=0,
	 &&(x,s)\in\Omega\times(0,t],\\
		&\phi(x,t)=\varphi(x),&&x\in\Om,
	\end{aligned}
	\right.
\end{equation}	
and  the  estimate	\begin{equation}\label{eq:phi-est}
		\|\nabla \phi(\cdot,s)\|_{L^{4/3}}
		\le
		C e^{C|\log(t/s)|^{1/2}}\|\nabla \varphi\|_{L^{4/3}},
		\qquad \forall \,s\in 
 (0,t].
	\end{equation}

{ First test} $\eqref{eq:delta-system-bdd}_1$ by $\phi$ { on $[\varepsilon,t]$. Since the flow preserves the boundary} and { $\|\phi(\varepsilon)\|_\infty\le\|\varphi\|_\infty$}, {this} {gives}
\begingroup
\[
\langle\rho^\delta(t),\varphi\rangle
=\langle\rho^\delta(\varepsilon),\phi(\varepsilon)\rangle
 +\int_\varepsilon^t\!\int_\Om\rho u^\delta\cdot\nabla\phi.
\]
\endgroup
{ The first term tends to zero by Definition \ref{d1}(i), because the two densities have the same initial value. Moreover $\rho u^\delta\in L^4(0,t;L^4)$ by the energy bounds and the Gagliardo-Nirenberg inequality, and}
\begingroup
\[
\int_0^t e^{\frac43 C|\log(t/s)|^{1/2}}\,ds
=t\int_0^\infty e^{-r+\frac43 C\sqrt r}\,dr<\infty.
\]
\endgroup
{ Thus \eqref{eq:phi-est} and H\"older's inequality permit $\varepsilon\downarrow0$, yielding}
\begingroup 
\[
\langle\rho^\delta(t),\varphi\rangle
=\int_0^t\!\int_\Om\rho u^\delta\cdot\nabla\phi,
\]
\endgroup
{  which,} together with \eqref{eq:phi-est}, { gives}
	\[
	|\langle \rho^\delta(\cdot,t),\varphi\rangle|
	\le
	C \|\nabla \varphi\|_{L^{4/3}}
	\int_0^t \|\rho u^\delta(\cdot,s)\|_{L^4}
	e^{C|\log(t/s)|^{1/2}}\,ds.
	\]
	Thus, 
	\[
	t^{-1}\|\rho^\delta(\cdot,t)\|_{W^{-1,4}}
	\le
	C F(t),\qquad \forall t\in(0,T],
	\]	
	where
	\[
	F(t)\triangleq\frac1t \int_0^t \|\rho u^\delta(\cdot,s)\|_{L^4}
	e^{C|\log(t/s)|^{1/2}}\,ds.
	\]
Combining this with	Lemma~\ref{lem:kernel} indicates 
	\begin{equation*} 
	\|s^{-1}\rho^\delta(\cdot,s)\|_{L^4(0,t;W^{-1,4})}
	\le
	C \|\rho u^\delta\|_{L^4(0,t;L^4)}.	
	\end{equation*}
	
Moreover, it follows from  \eqref{udrho} and the Gagliardo-Nirenberg inequality that
\begin{equation}
	\begin{aligned}\la{rhodeltau}
	\|\rho u^\delta\|_{L^4(0,t;L^4)}
	&\le \bar{\rho}\|u^\delta\|_{L^4(0,t;L^4)}\\
	&\le
	C \|u^\delta\|_{L^\infty(0,t;L^2)}^{1/2}
	\|\nabla u^\delta\|_{L^2(0,t;L^2)}^{1/2}\\
	&\le
	C \|\sqrt{\rho}u^\delta\|_{L^\infty(0,t;L^2)}+
	C\|\nabla u^\delta\|_{L^2(0,t;L^2)}\\&\le CD(t).
\end{aligned}
\end{equation}	
Then  
\begin{equation}\la{1111}
	\|s^{-1}\rho^\delta(\cdot,s)\|_{L^4(0,t;W^{-1,4})}
	\le
	CD(t).	
	\end{equation}

 Since  $u^\delta|_{\partial\Omega}=0$,  
$s\dot{\hat u}\cdot u^\delta\in W^{1,4/3}_0(\Omega)$ for a.e. $s\in[0,t]$. Then,   
\eqref{eq:strong-regularity-bdd},    \eqref{rhodeltau}, and   the  Poincar\'e  inequality imply
\begin{equation}
	\begin{aligned}
		\|s&\,\dot{\hat{u}}\cdot u^\delta\|_{L^{4/3}(\tau,t;W^{1,4/3}_0)}\\
		\le&
		C\|s\nabla\dot{\hat u}\|_{L^2(\tau,t;L^2)}
		\|u^\delta\|_{L^4(\tau,t;L^4)}
		\\&+C\|s\dot{\hat u}\|_{L^4(\tau,t;L^4)}
		\|\nabla u^\delta\|_{L^2(\tau,t;L^2)}\\
		\leq& C D(t)\left(
		\|s\nabla\dot{\hat u}\|_{L^2(\tau,t;L^2)}
		+\|s\dot{\hat u}\|_{L^4(\tau,t;L^4)}
		\right)\\
	\leq	& CD(t)\left(C_0^{1/2}\|s\nabla\dot{\hat u}\|_{L^2(\tau,t;L^2)}^{1/2}+CC_0^{1/2}
		\|s\nabla\dot{\hat u}\|_{L^2(\tau,t;L^2)}^{1/2}\right)\\
	\leq 	&C(C_0)D(t)
		\|s\nabla\dot{\hat u}\|_{L^2(\tau,t;L^2)}^{1/2},
		\label{eq:prod-est-1}
\end{aligned}\end{equation}
where in the second inequality, we have used the following   fact: \[
\|s\dot{\hat u}\|_{L^4(\tau,t;L^4)}
\leq C\|s\dot{\hat u}\|_{L^\infty(\tau,t;L^2)}^{1/2}
\|s\nabla\dot{\hat u}\|_{L^2(\tau,t;L^2)}^{1/2}\leq CC_0^{1/2}
\|s\nabla\dot{\hat u}\|_{L^2(\tau,t;L^2)}^{1/2},
\] owing to  \eqref{eq:strong-regularity-bdd}.
	
We now are ready to estimate $\int_0^tL_2(s)\,ds$.

The duality between
	$W^{-1,4}(\Omega)$ and $W^{1,4/3}_0(\Omega)$, \eqref{1111}, and  \eqref{eq:prod-est-1} give
	\begin{equation}\label{eq:I-dual}
		\begin{aligned}
			\int_\tau^t |L_2(s)|\,ds
			&\le
			\|s^{-1}\rho^\delta\|_{L^4(\tau,t;W^{-1,4})}
			\,
			\|s\,\dot{\hat{u}}\cdot u^\delta\|_{L^{4/3}(\tau,t;W^{1,4/3}_0)}\\
			&\le C\|s\nabla\dot{\hat u}\|_{L^2(\tau,t;L^2)}^{1/2} D^2(t), 
		\end{aligned}
	\end{equation}
for any $0\leq\tau<t\leq T$.

  Then, choosing $\tau=0$ in \eqref{eq:I-dual}, and using  $D(T) \leq C\left\|\left(u_0, \hat{u}_0\right)\right\|_{L^2}< +\infty$, we have
   $$\int_0^{ {T}}|L_2( {s})|  {ds}<+\infty.$$
  	
  If $D(0)=0$, then by \eqref{eq:strong-regularity-bdd} and the absolute continuity of the Lebesgue integral,   there
	exists $t_0>0$ such that  in \eqref{eq:I-dual}
	\[
	C \|s\nabla \dot{\hat{u}}\|_{L^2(0, t_0;L^2)}^{1/2}\le \frac12.
	\]
	This, and  \eqref{eq:D-basic},    \eqref{eq:I-dual} give
	$D(t)\equiv 0$ on $[0,t_0]$. Repeating the same argument over  successive intervals yields	$D(t)\equiv 0$ on $[0,T]$,  and then  \eqref{eq:stability-bdd} holds. Hence,  we may assume that $D(0)>0$.	

We next claim that there exists a constant $\tilde{C}>0$, depending only on
	$\underline{\rho},\bar{\rho},\Omega,$ and $C_0$, such that
	\begin{equation}\label{eq:I-claim}
		\int_0^T |L_2(s)|\,ds \le \tilde C D^2(0).
	\end{equation}
	
    In fact, without loss of generality, we may assume that  $\int_0^T |L_2(s)|\,ds > 2 D^2(0)$. Let  
	$K\in\mathbb N$ be such that
	\begin{equation}\la{k}
		2^K \le \frac{\int_0^T |L_2(s)|\,ds}{D^2(0)}< 2^{K+1}.
	\end{equation}
	Then, we  choose an increasing sequence $\{t_k\}_{k=0}^K\subset [0,T]$ with $t_0=0$ and
	\begin{equation}\la{kk}
		\int_0^{t_k} |L_2(s)|\,ds = 2^k D^2(0),
		\qquad \forall\, 1\le k\le K,
	\end{equation}
	which implies
	\begin{equation}\label{kkk}
	\int_{t_{k-1}}^{t_k}|L_2(s)|\,ds \geq2^{k-1}D^2(0),
	\qquad 1\le k\le K.
	\end{equation}
Now, it follows from  \eqref{eq:D-basic} and \eqref{kk} that
	\[
	D^2(t_k)
	\le
	C\int_0^{t_k}|L_2(s)|\,ds + C D^2(0)
	\le
	3C\,2^k D^2(0),
	\qquad 1\le k\le K.
	\]
	Inserting this, \eqref{kkk} into \eqref{eq:I-dual} on $(t_{k-1},t_k)$, we arrive at
	\[
	2^{k-1}D^2(0)
	\le
	C D^2(t_k)\|s\nabla \dot{\hat{u}}\|_{L^2(t_{k-1},t_k;L^2)}^{1/2}
	\le
	3C^2 2^k D^2(0)\|s\nabla \dot{\hat{u}}\|_{L^2(t_{k-1},t_k;L^2)}^{1/2}.
	\]
	Hence,
	\[
	\|s\nabla \dot{\hat{u}}\|_{L^2(t_{k-1},t_k;L^2)}
	\ge
	\frac1{36 C^4},
	\qquad 1\le k\le K.
	\]
	In view of \eqref{eq:strong-regularity-bdd}, one has
	\[
	\frac{K}{(36C^4)^2}
	\le
	\sum_{k=1}^K\|s\nabla \dot{\hat{u}}\|_{L^2(t_{k-1},t_k;L^2)}^2
	\le
	\|t\nabla \dot{\hat{u}}\|_{L^2(0,T;L^2)}^2
	\le C_0^2,
	\]
	i.e., $K\le (36C^4)^2 C_0^2$. This, and \eqref{k} give \eqref{eq:I-claim}.
	
	Finally, \eqref{eq:D-basic} and \eqref{eq:I-claim} show
	\[
	D(T)^2\le C D(0)^2.
	\]
So,
	\[
	D(T)\le C_1 D(0),
	\]
	for some constant $C_1=C_1(\underline\rho,\bar\rho,\Omega,C_0)>0$. Combining this with
	\eqref{eq:D-def}, we conclude that
	\[
	\sup_{0\le t\le T}\|\sqrt{\rho}\,u^\delta(\cdot,t)\|_{L^2}
	+
	\|\nabla u^\delta(\cdot,t)\|_{L^2(0,T;L^2)}
	\le
	C_1 D(0),
	\]
	which, together with  \eqref{1111} implies \eqref{eq:stability-bdd}. The proof of Lemma \ref{prop:stability-bdd} is finished.  \hfill $\Box$
	
With Lemma \ref{prop:stability-bdd} at hand,  we  now proceed to prove the  uniqueness of weak solutions in bounded domains by approximation with strong
solutions.

 \noindent{\bf Proof of Theorem \ref{tm2}.}
 Assume  that $0<\rho_0\in L^\infty(\Om)$ is bounded away from zero and $u_0 \in L^2_\sigma$. Let $(\rho^{(1)},u^{(1)},P^{(1)})$ and $(\rho^{(2)},u^{(2)},P^{(2)})$  (with $\int_\Om(P^{(1)},P^{(2)})dx=0$) be any two global weak solutions  with the same initial data $(\rho_0,u_0)$ satisfying the energy inequality \eqref{y}. By the renormalized transport equation,
\[
0<\underline\rho\leq\rho^{(j)}\leq\bar\rho,\qquad j=1,2.
\]

Let  $\{w_k\}_{k\ge 1}$ be an orthonormal basis of $L^2_\sigma$ consisting of the eigenfunction
of the Stokes operator $L=-\mathbb{P}\Delta$ in  $H_{0,\sigma}^1$
with   Dirichlet boundary condition. We write
\[
u_0=\sum_{k=1}^\infty a_k w_k
\quad\text{in }L^2_\sigma,\qquad
a_k=(u_0,w_k)_{L^2},
\]
and 
\[
u_0^N\triangleq\sum_{k=1}^N a_k w_k.
\]
Then 
\[
u_0^N\in H^1_{0,\sigma},\quad u_0^N\to u_0 \quad\text{in }L^2,
\quad
\|u_0^N\|_{L^2}\le C \|u_0\|_{L^2}.
\]

For every $N$,  there exists a  unique
 strong solution $(\rho^{(N)},u^{(N)})$ with initial data $(\rho_0,u_0^N)$ (see \cite[Theorem 2.1]{DanchinMucha2019}).  Since  \cite[Theorem 2.1]{Lions_1996} and \cite[Theorem 2.1]{DanchinMucha2019}, it satisfies
\[
0<\underline\rho\leq\rho^{(N)}\leq\bar\rho,\qquad
\sqrt{\rho^{(N)}} u_t^{(N)}\in L^2(0,T;L^2).
\]

By virtue of a priori estimates in Section \ref{3} with $(\rho, u, P)$ replaced by $(\rho^{(N)}, u^{(N)}, P^{(N)})$, we get
		\begin{equation*}
	\|\nabla  u^{(N)}\|_{L^2(0,T;L^2)}
	+\|t^{1/2}\nabla   u^{(N)}\|_{L^2(0,T;L^\infty)}
	+\|t\dot{  u}^{(N)}\|_{L^\infty(0,T;L^2)}
	+\|t\nabla \dot{  u}^{(N)}\|_{L^2(0,T;L^2)}
	\le C_0,
\end{equation*} 
where  $C_0 > 0$  depends  only on $\underline{\rho}$, $\bar{\rho}$, $\|\sqrt{\rho_0}u_0\|_{L^2}$,  $\Om$,  
and   independent of $N$, due to the fact 
$\|\sqrt{\rho_0}u_0^N\|_{L^2}\leq C\|u_0\|_{L^2}\leq C\|\sqrt{\rho_0}u_0\|_{L^2}.$
 
Furthermore,  applying Lemma  \ref{prop:stability-bdd} to  the pairs
$(\rho^{(j)},u^{(j)})$ and $(\rho^{(N)},u^{(N)})$, $j=1,2$, 
\begin{align}
	&\sup_{0\le t\le T}\|\sqrt{\rho^{(j)}}\,(u^{(j)}-u^{(N)})(t)\|_{L^2}
	+
	\|\nabla(u^{(j)}-u^{(N)})\|_{L^2(0,T;L^2)}\notag\\
	&\hspace{4cm}
	+
	\|t^{-1}(\rho^{(j)}-\rho^{(N)})\|_{L^4(0,T;W^{-1,4})}
	\le
	C \|u_0-u_0^N\|_{L^2},
	\qquad j=1,2,
	\label{eq:approx-stab}
\end{align}
where $C>0$ is independent of $N$. Letting $N\to\infty$ in \eqref{eq:approx-stab} and by the positive lower density bound, we derive
\[
u^{(1)}=u^{(2)}
\quad\text{in }L^\infty(0,T;L^2)\cap L^2(0,T;H^1_0),
\]
and  for any $\tau\in(0,T)$,
\[
\rho^{(1)}=\rho^{(2)}
\quad\text{in }L^4(\tau,T;W^{-1,4}).
\]
This proves the uniqueness of the weak solution
$(\rho,u)$.  Finally, subtracting the two momentum equations,  gives
$\nabla(P^{(1)}-P^{(2)})=0$ in distributions.  This, combined with     $\int_\Om P^{(j)}dx=0, j=1,2$ for a.e. $t\in[0,T]$,   implies
$P^{(1)}=P^{(2)}$ a.e. $t\in\Om\times[0,T]$. So far, we finish the proof of  the  uniqueness  in Theorem \ref{tm2}. \hfill $\Box$

\appendix
\setcounter{section}{0}
\refstepcounter{section}
\section*{Appendix \thesection: Some technical lemmas}
\label{sec:technical-lemmas}
\addcontentsline{toc}{section}{Appendix \thesection: Some technical lemmas}
In this section, we compile several auxiliary results  used in Sections \ref{sec:bdry-linfty} and \ref{4.2}.

First, we give the   construction of  bounded $C^2$ smooth domains near the boundary.
\begin{lemma}\label{lm:geompatch}
	Let $\Omega\subset\mathbb R^2$ be a bounded $C^2$ domain. Then there exist
	$r_0>0$ and $C>0$, depending only on $\Omega$, such that for every $x_b\in\partial\Omega$,
	and  $0<r\le r_0$, there exists a bounded $C^2$ domain
	$D_{4r}(x_b)\subset\Omega$ satisfying
\begin{equation}\label{geompatch4}
		D_{4r}(x_b)\cap B(x_b,2r)=\Omega\cap B(x_b,2r),
	\end{equation}
	and
\begin{equation}\label{geompatch5}
		\Omega\cap B(x_b,r)\subset D_{4r}(x_b)\subset \Omega\cap B(x_b,6r).
	\end{equation}
	Moreover, the normalized 
	\begin{equation}\label{geompatch6}
		\widetilde D_4(x_b,r)\triangleq r^{-1}T_{x_b}\bigl(D_{4r}(x_b)\bigr)
	\end{equation}
    have uniformly bounded \(C^2\) boundary and uniformly bounded   measure,  with both bounds independent of \(x_b\in\partial\Omega\) and \(0<r\le r_0\),  where $T_{x_b}$ represents  a rotation  and  translation transformation. 
	In addition,
	\begin{equation}\label{geompatch-area}
		C_2r^2\leq|D_{4r}(x_b)|\leq C_3r^2,
	\end{equation}
		where the   constants $C_2$ and $C_3$ depend only on $\Omega$.
\end{lemma}
\noindent{\it Proof.} Since $\Omega$ is a bounded $C^2$ domain, there exist
constants $r_*>0$ and $M>0$, depending only on $\Omega$, such that for every
$x_b\in\partial\Omega$, there exist a     rotation and translational transformation 
\[
T_{x_b}(x)\triangleq Q_{x_b}(x-x_b),
\] where $Q_{x_b}$ is a \(2\)-dimensional orthogonal matrix, 
such that the tangent line to $\partial\Om$ at $x_b$ is mapped to the
$y_1$-axis,  and a function
\(\psi_{x_b}\in C^2((-6r_*,6r_*))\) satisfying
\begin{equation}\la{iiii}
	\psi_{x_b}(0)=0,\quad \psi_{x_b}'(0)=0,\quad
	\|\psi_{x_b}\|_{C^2((-6r_*,6r_*))}\le M,
\end{equation}
and
\[
T_{x_b}\bigl(\Omega\cap B(x_b,6r_*)\bigr)
=
\{y=(y_1,y_2)\in B(0,6r_*): y_2>\psi_{x_b}(y_1)\}.
\]

	For $0<r\le r_*$, we define the rescaled boundary graph
\[
\gamma_{x_b,r}(z_1)\triangleq\frac{1}{r}\psi_{x_b}(rz_1),
\quad |z_1|<6,
\]
which, together with \eqref{iiii} and Taylor's formula, gives for $|z_1|< 6$,
\[
|\gamma_{x_b,r}(z_1)|\le 36M r,\quad
|\gamma_{x_b,r}'(z_1)|\le 6M r,\quad
|\gamma_{x_b,r}''(z_1)|\le M r.
\]
Denoting \(r_0
\triangleq
\min\left\{
r_*,
\frac{1}{72M},
1
\right\},\)
we get, for all $ x_b\in\partial\Omega,$  and $ 0<r\le r_0$,
\begin{equation}\label{gamma-small}
	\|\gamma_{x_b,r}\|_{L^\infty(-6,6)}\le \frac12.
\end{equation}
	
Straightening  out the boundary, we set 
\[
\Phi_{x_b,r}(z_1,z_2)
\triangleq
\bigl(z_1,z_2+\gamma_{x_b,r}(z_1)\bigr),
\] for $(z_1,z_2)\in B(0,6)\cap\{z_2>0\}$. 
Then $\Phi_{x_b,r}$ is a $C^2$ diffeomorphism and
\[
\Phi_{x_b,r}^{-1}(Y_1,Y_2)
=
\bigl(Y_1,Y_2-\gamma_{x_b,r}(Y_1)\bigr),
\quad
\det D\Phi_{x_b,r}\equiv 1.
\]
	Moreover, the $C^2$ norms of $\Phi_{x_b,r}$ and $\Phi_{x_b,r}^{-1}$ are bounded
uniformly with respect to $x_b$  and $r$.

	We now construct $D_{4r}(x_b)$. 
    
    First, fix  a bounded
	$C^\infty$ domain $\mathcal Q\subset\mathbb R^2$ satisfying
	\begin{equation}\label{Q-property}
		\mathcal Q\cap B(0,3)=B(0,3)\cap\{z_2>0\},
		\quad
		\mathcal Q\subset B(0,5)\cap\{z_2>0\}.
	\end{equation}
Such a domain $\mathcal Q$ follows from a standard smooth cut-off construction
for the boundary outside \(B(0,3)\). 
	
	Set
	\begin{equation}\label{Dtilde-def}
		\widetilde D_4(x_b,r)\triangleq\Phi_{x_b,r}(\mathcal Q),
	\end{equation}
	and 
	\begin{equation}\label{D-def}
		D_{4r}(x_b)\triangleq
		T_{x_b}^{-1}\bigl(r\widetilde D_4(x_b,r)\bigr),
	\end{equation}
    which implies \eqref{geompatch6}.
Note that  $\mathcal Q$ is a   bounded $C^\infty$ domain and $\Phi_{x_b,r}$ is a
	$C^2$ diffeomorphism, then     $\widetilde D_4(x_b,r)$  are uniformly 
	bounded $C^2$ domains with respect to $x_b$ and $r$.
	
	Let
	\[
	\Omega_{x_b,r}\triangleq r^{-1}T_{x_b}(\Omega).
	\]
Then
	\[
	\Omega_{x_b,r}\cap B(0,6)
	=
	\{Y=(Y_1,Y_2)\in B(0,6):Y_2>\gamma_{x_b,r}(Y_1)\}.
	\]
	
We first claim  that
	\[
	\widetilde D_4(x_b,r)\subset \Omega_{x_b,r}\cap B(0,6).
	\]
	Indeed, if $Y=\Phi_{x_b,r}(z)$ with $z\in\mathcal Q$, then 
	by $\mathcal Q\subset\{z_2>0\}$,
	\[
	Y_2=z_2+\gamma_{x_b,r}(z_1)>\gamma_{x_b,r}(Y_1),
	\]
which gives  $Y\in\Omega_{x_b,r}$. On the other hand, by
	$\mathcal Q\subset B(0,5)$ and \eqref{gamma-small}, we have
	\[
 |Y|\le|z|+|\gamma_{x_b,r}(z_1)|<5+\frac12<6.
\]
	Thus, 
	\begin{equation}\label{Dtilde-subset}
		\widetilde D_4(x_b,r)\subset\Omega_{x_b,r}\cap B(0,6).
	\end{equation}
	
	Next,  we prove 
	\begin{equation}\label{Dtilde-local-equality}
		\widetilde D_4(x_b,r)\cap B(0,2)
		=
		\Omega_{x_b,r}\cap B(0,2).
	\end{equation}
	The inclusion ``$\subset$'' follows directly from \eqref{Dtilde-subset}. Conversely,
	let $Y\in\Omega_{x_b,r}\cap B(0,2)$ and 
	\[
	z\triangleq\Phi_{x_b,r}^{-1}(Y)
	=
	\bigl(Y_1,Y_2-\gamma_{x_b,r}(Y_1)\bigr).
	\]
Since $Y\in\Omega_{x_b,r}$, we observe $z_2>0$. Moreover, 
	\eqref{gamma-small} leads to 
		\[
	|z|\le|Y|+|\gamma_{x_b,r}(Y_1)|<2+\frac12<3,
	\]
	which gives
	\[
	z\in B(0,3)\cap\{z_2>0\}
	=
	\mathcal Q\cap B(0,3).
	\]
	Hence $Y=\Phi_{x_b,r}(z)\in\widetilde D_4(x_b,r)$, which yields
	\eqref{Dtilde-local-equality}. Also, by \eqref{D-def}-\eqref{Dtilde-local-equality}, we can check \eqref{geompatch4}-\eqref{geompatch5}.
	
	Finally, since \eqref{Dtilde-def}, \eqref{D-def}, and 
 $\det D\Phi_{x_b,r}\equiv1$, 
	we have
	\[
	|D_{4r}(x_b)|
	=
	r^2|\widetilde D_4(x_b,r)|
	=
	r^2|\mathcal Q|.
	\]
	This proves \eqref{geompatch-area}. So, we finish the proof of Lemma \ref{lm:geompatch}.
 \hfill $\Box$

The next lemma establishes a generalized  Poincar\'e inequality, which remains valid even when the  function vanishes solely on a prescribed portion of  boundary.
\begin{lemma}
	\label{lm:local-poincare-normalized}
	Let
	\[
	\widetilde\Gamma_{x_b,r}
	\triangleq
	\partial\widetilde D_4(x_b,r)\cap B(0,2).
	\]
	Then there exists a constant \(C>0\), depending only on \(\Omega\) and  the fixed  domain
	\(\mathcal Q\), such that for any
	\(f\in H^1(\widetilde D_4(x_b,r))\) satisfying
	\[
	f=0
	\quad\text{on } \widetilde\Gamma_{x_b,r},
	\]
	one has
	\[
	\|f\|_{L^2(\widetilde D_4(x_b,r))}
	\le
	C\|\nabla f\|_{L^2(\widetilde D_4(x_b,r))},
	\]
where \(\mathcal Q\) and $\widetilde D_4(x_b,r)$ are defined in Lemma \ref{lm:geompatch}.\end{lemma}
\noindent{\it Proof.}
By the construction of the domain \(\mathcal Q\) in \eqref{Q-property}, the flat segment
	\[
	\Gamma_*\triangleq\{(z_1,0): |z_1|<1\}
	\subset
\partial\mathcal Q.\]
Moreover,
	\[
	\Phi_{x_b,r}(z_1,0)
	=
	(z_1,\gamma_{x_b,r}(z_1)).
	\]
	By  \(\|\gamma_{x_b,r}\|_{L^\infty([-1,1])}\le 1/2\), we have
	\[
	|\Phi_{x_b,r}(z_1,0)|
	\le
	\sqrt{1+\frac14}
	<2 ,
	\qquad |z_1|<1.\]
Then, 
	\[
	\Phi_{x_b,r}(\Gamma_*)
	\subset
	\partial\widetilde D_4(x_b,r)\cap B(0,2)
	=
	\widetilde\Gamma_{x_b,r}.
	\]
	  Thus $\widetilde\Gamma_{x_b,r}$ contains  the   image   of the fixed segment $\Gamma_*$.
	
	Denote 
	\[
	g(z)\triangleq f(\Phi_{x_b,r}(z)),
	\quad z\in \mathcal Q,
	\] which, together with  \(f=0\) on \(\widetilde\Gamma_{x_b,r}\), gives
	\[
	g=0
	\quad\text{on } \Gamma_*.
	\]
	
	We now claim that 
	there exists \(C({\mathcal Q})>0\) such that
	\begin{equation}\la{hy8}
	\|g\|_{L^2(\mathcal Q)}
\le
C({\mathcal Q})\|\nabla g\|_{L^2(\mathcal Q)},
	\end{equation}
	for any \(g\in H^1(\mathcal Q)\) and \(g=0\) on \(\Gamma_*\).
	
	Indeed, arguing by contradiction, we see  that there exists a sequence 	$\{g_n\} \subset H^1(\mathcal Q)$ satisfying
	\[
	g_n=0\,\, \text{ on } \Gamma_*,
	\quad
	\|g_n\|_{L^2(\mathcal Q)}=1,
	\quad
	\|\nabla g_n\|_{L^2(\mathcal Q)}\to 0,\,\, \text{as}\,\, n\rightarrow \infty.
	\]
	Then   \(\{g_n\}\) is bounded in \(H^1(\mathcal Q)\).  Since \(\mathcal Q\) is a bounded   \(C^2\) domain,
	the compact embedding \(H^1(\mathcal Q)\hookrightarrow L^2(\mathcal Q)\) indicates, up to a subsequence,
	\[
	g_n\to g
	\quad\text{strongly in } L^2(\mathcal Q),
	\quad
	g_n\rightharpoonup g
	\quad\text{weakly in } H^1(\mathcal Q).
	\]
	Moreover, since \(\nabla g_n\to 0\) in \(L^2(\mathcal Q)\),  \(\nabla g=0\)   i.e. \(g\) is a constant on \(\mathcal Q\).
	
On the other hand, it follows from the continuous embedding  $H^1(\mathcal Q)\hookrightarrow
	H^{1/2}(\partial\mathcal Q)$, and \(g_n=0\) on \(\Gamma_*\) that  passing to the limit,
	\[
	g=0
	\quad\text{on } \Gamma_*, 
	\]
	which, combined with  the fact that \(g\) is constant, yields  \(g\equiv 0\).
	This contradicts
	\[
	\|g\|_{L^2(\mathcal Q)}
	=
	\lim_{n\to\infty}\|g_n\|_{L^2(\mathcal Q)}
	=
	1 .
	\]
	Hence \eqref{hy8} holds.
	
	  Since
	\[
	\Phi_{x_b,r}(z_1,z_2)
	=
	(z_1,z_2+\gamma_{x_b,r}(z_1)),
	\]
	we have
	\[
	D\Phi_{x_b,r}
	=
	\begin{pmatrix}
		1 & 0\\
		\gamma_{x_b,r}'(z_1) & 1
	\end{pmatrix},
	\quad
	\det D\Phi_{x_b,r}=1,
	\]which gives
	\begin{equation}\la{iiiii}
			\|f\|_{L^2(\widetilde D_4(x_b,r))}
		=
		\|g\|_{L^2(\mathcal Q)}.
	\end{equation}
	Moreover,  the uniform boundedness of \(\gamma_{x_b,r}\) in \(C^2\), independent of  \(x_b\) and \(r\), implies 
	\[
	\|D\Phi_{x_b,r}\|_{L^\infty(\mathcal Q)}
	+
	\|D\Phi_{x_b,r}^{-1}\|_{L^\infty(\widetilde D_4(x_b,r))}
	\le
	C(\Omega).
	\]
	This and 
	\[
	\nabla g(z)
	=
	D\Phi_{x_b,r}(z)^T
	\nabla f(\Phi_{x_b,r}(z)),
	\]
lead  to
	\[
	\|\nabla g\|_{L^2(\mathcal Q)}
	\le
	C(\Omega)
	\|\nabla f\|_{L^2(\widetilde D_4(x_b,r))},
	\]
	which, combined with  \eqref{hy8} and \eqref{iiiii}, gives
	\[
	\|f\|_{L^2(\widetilde D_4(x_b,r))}
	=
	\|g\|_{L^2(\mathcal Q)}
	\le
	C({\mathcal Q})\|\nabla g\|_{L^2(\mathcal Q)}
	\le
	C(\Omega,{\mathcal Q})
	\|\nabla f\|_{L^2(\widetilde D_4(x_b,r))}.
	\] Thus, the proof of Lemma \ref{lm:local-poincare-normalized} is completed.
 \hfill $\Box$

 \section*{Data Availability}
 
 Data sharing is not applicable to this article as no new data were created or analyzed in this study.

\end{document}